\documentclass[reqno]{amsart}
\usepackage{amsmath,amsthm,amsfonts, amssymb,amscd,enumitem}
\usepackage[all]{xy}
\numberwithin{equation}{section}
  \newtheorem{thm}{Theorem}[section]
  \newtheorem{lem}[thm]{Lemma}
  \newtheorem{prop}[thm]{Proposition}
  \newtheorem{cor}[thm]{Corollary}
  \theoremstyle{definition}
  \newtheorem{defn}[thm]{Definition}
  \newtheorem{exm}[thm]{Example}
  \newtheorem{rmk}[thm]{Remark}
  \newtheorem{open}[thm]{Problem}

 \newcommand\ra{\rightarrow}

\newcommand{\lex}{\,\overrightarrow{\times}\,}

 \newcommand\s{\subseteq}

 \newcommand\B{\mathrm{B}}

\newcommand{\Ker}{\mbox{\rm Ker}}
\newcommand{\id}{\mbox{\rm Id}}

\newcommand\C{\mathrm{C}}

 \numberwithin{equation}{section}

 \def\iff{if and only if }

\let\Right\right
\let\Left\left
\makeatletter
\def\right#1{\Right#1\@ifnextchar){\!\right}{}}
\def\left#1{\Left#1\@ifnextchar({\!\left}{}}
\begin{document}
\title[Representation and Embedding of Pseudo MV-algebras with Square Roots II]{Representation and Embedding of Pseudo MV-algebras with Square Roots II. Closures}

\author[Anatolij Dvure\v{c}enskij and Omid Zahiri]{Anatolij Dvure\v{c}enskij$^{^{1,2,3}}$, Omid Zahiri$^{^{1,*}}$}

\date{}%
\thanks{The paper acknowledges the support by the grant of
the Slovak Research and Development Agency under contract APVV-20-0069  and the grant VEGA No. 2/0142/20 SAV,  A.D}
\thanks{The project was also funded by the European Union's Horizon 2020 Research and Innovation Programme on the basis of the Grant Agreement under the Marie Sk\l odowska-Curie funding scheme No. 945478 - SASPRO 2, project 1048/01/01,  O.Z}
\thanks{* Corresponding Author: Omid Zahiri}
\address{$^1$Mathematical Institute, Slovak Academy of Sciences, \v{S}tef\'anikova 49, SK-814 73 Bratislava, Slovakia}
\address{$^2$Palack\' y University Olomouc, Faculty of Sciences, t\v r. 17. listopadu 12, CZ-771 46 Olomouc, Czech Republic}
\address{$^3$Depart. Math., Constantine the Philosopher University in Nitra, Tr. A. Hlinku 1, SK-949 01 Nitra, Slovakia}
\email{dvurecen@mat.savba.sk, zahiri@protonmail.com}
\thanks{}

\keywords{Pseudo MV-algebra, unital $\ell$-group, symmetric pseudo MV-algebra, square root, strict square root, square root closure, two-divisibility, divisibility, embedding}
\subjclass[2020]{06C15, 06D35}


\begin{abstract}
In \cite{DvZa3}, we started the investigation of pseudo MV-algebras with square roots. In the present paper, the main aim is to continue to study the structure of pseudo MV-algebras with square roots focusing on their new characterizations.
The paper is divided into two parts. In the first part, we investigate the relationship between a pseudo MV-algebra with square root and its corresponding unital $\ell$-group in the scene of two-divisibility.

In the present second part, we find some conditions under which a particular class of pseudo MV-algebras can be embedded into pseudo MV-algebras with square roots. We introduce and investigate the concepts of a strict square root of a pseudo MV-algebra and a square root closure, and we compare both notions. We show that each MV-algebra has a square root closure. Finally, using the square root of individual elements of a pseudo MV-algebra, we find the greatest subalgebra of a special pseudo MV-algebra with weak square root.
\end{abstract}

\date{}

\maketitle

\section*{Introduction}

In this paper, we continue with the study from \cite{DvZa4}, where we have studied square roots on pseudo MV-algebras. Square roots on pseudo MV-algebra were generalized from ones on MV-algebras in \cite{DvZa3}. Motivation for studying square roots and strict square roots and their basic properties can be found in \cite{DvZa3} together with the necessary notions and definitions. In this part, sections, theorems, propositions, lemmas, examples, and equations are numbered in continuation of \cite{DvZa4}.

\section{Embedding and strict square root closure}

In this section, we  study the possibility of embedding a pseudo MV-algebra $M$ into a pseudo MV-algebra with strict square root. Moreover, we introduce a strict square root closure of a pseudo MV-algebra.

First, we recall some notes about divisible (pseudo) MV-algebras. Let $M=\Gamma(G,u)$ be a pseudo MV-algebra, where $(G,u)$ is a
unital $\ell$-group.

\subsection{Commutative case}
Every MV-algebra $M=\Gamma(G,u)$ can be embedded into a divisible one: Let $G^d$ be the group-theoretic divisible hull of the Abelian group $G$. Then $G^d$ can be converted into an $\ell$-group assuming that an element $h\in G^d$ is positive in $G^d$ if $nh\in G^+$ for some integer $n\ge 1$ (see \cite[P. 23]{Anderson}). If $G^d$ is the divisible hull of the $\ell$-group $G$, then $G^d$ is an $\ell$-group with strong unit $u$,
and we use $M^d$ to denote the MV-algebra $\Gamma(G^d,u)$.
If $M$ is a linearly ordered MV-algebra, then $M^d$ is linearly ordered, too (see \cite[Cor 1.4.26]{DiLe}, \cite[Lem IV.5.A]{Fuc}). For more information about divisible MV-algebras, we suggest to see \cite{DiLe,DiSe,DvRi,LaLe}.

\begin{thm}\cite[Prop 5.2]{DvZa2}
Every MV-algebra can be embedded in an MV-algebra with square root.
\end{thm}

\subsection{Non-commutative case} For the non-commutative case, the situation is different.
Holland, \cite{Ho}, showed that every $\ell$-group could be embedded in a divisible $\ell$-group. He proved that every $\ell$-group $G$ is an $\ell$-subgroup of $A(\Omega)$ for some suitable totally ordered set $\Omega$, where $A(\Omega)$ is the set of all order-preserving bijective maps on $\Omega$.
But, if $G$ has a strong unit, there is no guarantee that the divisible $\ell$-group $A(\Omega)$ has a strong unit.

On the other side, let $M=\Gamma(G,u)$ be an arbitrary pseudo MV-algebra. Take the divisible $\ell$-group $A(\Omega)$ in which $G$ can be embedded; for simplicity, we can assume $G\subseteq A(\Omega)$. Set $G(u)=\{g\in A(\Omega)\mid g\le nu$ for some integer $n\ge 1\}$. Then $G(u)=\{g\in A(\Omega)\mid |g|\le mu$ for some integer $m\ge 1\}$ and $(G(u),u)$ is a unital $\ell$-group. Let $g\in \Gamma(G(u),u)$. Given $n\ge 1$, there exists $h\in A(\Omega)$ such that $nh=g$, so that $nh\ge 0$ which implies $h\ge 0$ (see \cite[Prop 3.6]{Dar}), and $0\le h\le nh\le u$. Therefore, $h\in \Gamma(G(u),u)$, proving that $M=\Gamma(G,u)$ can be embedded into a divisible pseudo MV-algebra $\Gamma(G(u),u)$.

Using similar reasoning, it is possible to show that also $G(u)$ is divisible: Let $g\in G(u)$, there is an integer $m\ge 1$ such that $|g|\le mu$. Given $n\ge 1$, there is an element $h\in A(\Omega)$ such that $nh=g$. Then $|h|\le n|h|= |nh|=|g|\le mu$, implying $h\in G(u)$. Consequently, every unital $\ell$-group $(G,u)$ can be embedded into a divisible unital $\ell$-group. We note that there is no guarantee that $G(u)$ enjoys unique extraction of roots.

In the sequel, we find some classes of unital $\ell$-groups that can be embedded into divisible (two-divisible) unital $\ell$-groups.

The class of $\ell$-groups that can be embedded into a representable two-divisible $\ell$-group will be denoted by $\mathcal{RDG}$. The class  $\mathcal{RDG}$ contains the class of MV-algebras. This class is closed under direct products and subalgebras. In Example \ref{div001}, we will see some $\ell$-groups belonging to $\mathcal{RDG}$.

Recall that, in \cite{DvZa3}, we found some results about symmetric representing pseudo MV-algebras with
square root with the additional condition $u/2\in \C(G)$. The following lemma shows that the assumption $u/2\in \C(G)$ is superfluous. In each symmetric representing pseudo MV-algebra $M=\Gamma(G,u)$ with square root, if $u/2$ exists, it always belongs to the center of $G$.

\begin{lem}\label{div01}
Let $M=\Gamma(G,u)$ be a pseudo MV-algebra, where $(G,u)$ is a unital representable $\ell$-group.
\begin{itemize}[nolistsep]
\item[{\rm (i)}] If $M$ has a square root, then $M$ is symmetric. In addition, if $u/2$ is defined, then $u/2\in \C(G)$.

\item[{\rm (ii)}] If $G$ is two-divisible and $u\in \C(G)$, then $r:M\to M$ defined by $r(x)=(x-u)/2+u$ is a strict square root on $M$.
Moreover, $u/2\in \C(G)$.
\end{itemize}
\end{lem}

\begin{proof}
(i) Let $r:M\to M$ be a square root on $M$. Consider the subdirect embedding $f:M\to \prod_{i\in I}M_i$, where each $M_i$ is linearly ordered. For each $i\in I$, the onto homomorphism $\pi_i\circ f:M\to M_i$ and \cite[Prop 3.9]{DvZa3} imply that
$M_i$ has a square root $r_i$ induced by $\pi_i\circ f$ and by $r$, where $\pi_i$ is the $i$-th natural projection map.
If $|M_i|=2$, then clearly $M_i$ is symmetric. Otherwise, the only Boolean elements of $M_i$ are $0$ and $1$, so by Proposition 2.4(11)(b), $r_i(0)^-\odot r_i(0)^-=0$ or $r_i(0)^-\odot r_i(0)^-=1$. The equality
$r_i(0)^-\odot r_i(0)^-=0$ implies that $r_i(0)^-\leq r_i(0)$ consequently, $r_i(0)^-=r_i(0)$ (\cite[(4.1)]{DvZa3}) which means
$r_i$ is strict and so $M_i$ is symmetric by \cite[Thm 5.6]{DvZa3}.
If $r_i(0)^-\odot r_i(0)^-=1$, then $r_i(0)=0$. It follows from \cite[Thm 3.8]{DvZa3} that $M_i$ is a Boolean algebra which is
absurd (since $M_i$ is a chain with more than $2$ elements).
Hence, for each $i\in I$, the pseudo MV-algebra $M_i$ is symmetric; consequently, $\prod_{i\in I}M_i$ is symmetric, too.
Therefore, $M$ is symmetric.

Let $u/2$ be defined. Assume that $M$ is a chain. We claim that $u/2\in \C(G)$. Let $g\in G$. Without loss of generality, we assume that
$u/2+g\leq g+u/2$. Then $g+u=u+g=u/2+u/2+g\leq u/2+g+u/2$ and so $g+u/2\leq u/2+g$. That is, $u/2+g=g+u/2$ which means $u/2\in \C(G)$.

Now, let $M$ be representable; then so is the unital $\ell$-group $(G,u)$. Let $f:G\to \prod_{i\in I}G_i$ be a subdirect embedding where $\{(G_i,u_i)\mid i\in I\}$ is a family of linearly ordered unital $\ell$-groups.
On the other hand, each $M_i=\Gamma(G_i,u_i)$ is symmetric, where $u_i=\pi_i(f(u))$.
For each $i\in I$, we have $\pi_i(f(u/2))=(\pi_i(f(u)))/2=u_i/2$, since $\pi_i(f(u/2))+\pi_i(f(u/2))=\pi_i(f(u))$.
By the first step of the proof, $\pi_i(f(u/2))\in \C(G_i)$ for each $i\in I$.
Let $g\in G$. Then $\pi_i(f(u/2+g))=\pi_i(f(u/2))+\pi_i(f(g))=\pi_i(f(g))+\pi_i(f(u/2))$ for all $i\in I$ and so
$f(u/2+g)=f(g+u/2)$ which entails that $u/2+g=g+u/2$, since $f$ is one-to-one. Therefore, $u/2\in \C(G)$.

(ii)  Let $x,y\in M$. Then
$r(x)\odot r(x)=((x-u)/2+u-u+(x-u)/2+u)\vee 0=(x-u+u)\vee 0=x$. Also, if $y\odot y\leq x$, then
$y-u+y\leq (y-u+y)\vee 0\leq x$, so $2y-2u\leq x-u$, consequently $y-u\leq (x-u)/2$ and so $y\leq (x-u)/2+u=r(x)$ (since $G$ is representable and $u\in \C(G)$).
In addition, $r(0)=(0-u)/2+u=u/2=u-r(0)=r(0)^-$. The mapping $r$ is a weak strict square root and $r(0)=u/2$.
Now, (i) implies that $u/2\in \C(G)$ and so $r(x)=(x-u)/2+u=(x+u)/2$. Therefore, $r$ is a strict square root; see \cite[Exm 3.7(iii)]{DvZa3}.
\end{proof}

\begin{thm}\label{div1}
Let $(G;+,0)$ be an $\ell$-group belonging to $\mathcal{RDG}$ and $(H,u)$ be an Abelian linearly ordered unital $\ell$-group. The pseudo MV-algebra $M=\Gamma(H\lex G,(u,0))$ can be embedded in a pseudo MV-algebra with strict square root.
\end{thm}

\begin{proof}
If $|M|=2$, then $M$ is a Boolean algebra and by \cite[Thm 3.8]{DvZa3}, $\id_M$ is a square root on $M$. So, let $|M|\neq 2$.
Consider the $\ell$-group $H\lex G$.
Since $H$ is a linearly ordered Abelian $\ell$-group with strong unit $u$, there exists a linearly ordered divisible
$\ell$-group $H^d$ with the strong unit $w$ such that $(H,u)$ can be embedded in $(H^d,w)$ as a unital $\ell$-group.
Let $f:H\to H^d$ be the embedding with $f(u)=w$.
By the assumptions, a representable divisible $\ell$-group $D$ and an $\ell$-group embedding $g:G\to D$ exist. Consider the lexicographic product $H^d\lex D$.
The mapping $\alpha:H\lex G\to H^d\lex D$ sending $(x,y)$ to $(f(x),g(y))$ is an $\ell$-group embedding.
Clearly, $H^d\lex D$ is a representable two-divisible $\ell$-group with the strong unit $(w,0)$.
For each $(x,y)\in H^d\lex D$, $(x,y)/2=(x/2,y/2)$.
Also, $(w,0),(w,0)/2=(w/2,0)\in \C(H^d\lex D)$, so by Lemma \ref{div01}, $r:\Gamma(H^d\lex D)\to \Gamma(H^d\lex D)$
defined by $r(x,y)=((x,y)-(w,0))/2+(w,0)$ is a strict square root on $\Gamma(H^d\lex D)$.

On the other hand, a pseudo MV-algebra embedding $\beta:\Gamma(H\lex G,(u,0))\to \Gamma(H^d\lex D,(w,0))$ can be induced
by the unital $\ell$-group embedding $\alpha:H\lex G\to H^d\lex D$.
Therefore, $M=\Gamma(H\lex G,(u,0))$ can be embedded in a pseudo MV-algebra with square root.
In addition, by Lemma \ref{div01}, $r:\Gamma(H^d\lex D,(w,0))\to \Gamma(H^d\lex D,(w,0))$ is strict, and so $M$ can be embedded
in a pseudo MV-algebra with strict square root.
\end{proof}

As a special case of Theorem \ref{div1}, if $G$ is a representable divisible pseudo MV-algebra, then clearly
$M=\Gamma(H\lex G,(u,0))$ can be embedded in a pseudo MV-algebra with strict square root.

\begin{exm}\label{div001}
(i) Take the non-Abelian $\ell$-group $G_1:=\mathbb Z\lex\mathbb Z\lex\mathbb Z$ with a strong unit $u=(1,0,0)$, where $G_1$ is with the following operations for all $(a,b,c),(x,y,z)\in G_1$:
\begin{eqnarray*}
(a,b,c)+(x,y,z)=(a+x,b+y,c+z+ay),\ \ -(a,b,c,d)=(-a,-b,-c+ab).
\end{eqnarray*}
It is a linearly ordered and non-divisible $\ell$-group.
It can be embedded into linearly ordered divisible $\ell$-groups $G_2=\mathbb D\lex\mathbb D\lex\mathbb D$, $G_3=\mathbb Q\lex\mathbb Q\lex\mathbb Q$, and $G_4=\mathbb R\lex\mathbb R\lex\mathbb R$, respectively, where the group operations are defined analogously as for $G_1$, and $\mathbb D$ is the group of dyadic numbers in $\mathbb R$, i.e. $\mathbb D=\{m/2^n\mid  m\in \mathbb Z, n=0,1,\ldots\}$.

Therefore, $G_1$ is an element of $\mathcal{RDG}$.

(ii) The non-Abelian $\ell$-group $G_5:=\mathbb Q\lex\mathbb Q\lex\mathbb Q\lex\mathbb Q$ that was introduced in Example 4.11 is divisible and representable with a strong unit. So, it belongs to $\mathcal{RDG}$.

(iii) By Theorem \ref{div1}, if $H$ is a subgroup of $\mathbb R$, then $H\lex G_1$ and $H\lex G_2$ are elements of
$\mathcal{RDG}$.

(iv) A locally nilpotent group is a group whose finitely generated subgroups are nilpotent. An $\ell$-group is locally nilpotent if it is locally nilpotent as a group.
According to Mal'cev (see \cite[Thm 7.3.2]{KoMe}),
every locally nilpotent linearly ordered group $G$ can be embedded in a divisible locally nilpotent linearly
ordered group. Therefore, every locally nilpotent linearly ordered group $G$ belongs to $\mathcal{RDG}$. In addition, by Theorem \ref{div1}, $H\lex G\in \mathcal{RDG}$ for each Abelian unital $\ell$-group $(H,1)$.
\end{exm}

\begin{rmk}\label{div2}
Let $M=\Gamma(G,u)$ be an arbitrary linearly ordered pseudo MV-algebra.
Then $K:=\{g\in G\mid n|g|<u,~\forall n\in\mathbb N\}$ is the only maximal convex $\ell$-subgroup of $G$; see \cite[Prop 5.4, Thm 6.1]{Dvu2}. Then $I=K\cap [0,u]$ is a normal and maximal ideal of $M$. It follows that $M/I$ is an MV-algebra which is isomorphic to a subalgebra of the MV-algebra of the real unit interval $[0,1]$. We can show that $M$ is an $(H,1)$-perfect MV-algebra, where $(H,1)=\Gamma^{-1}(M/I)\cong (G/H,u/H)$.

Now, using Theorem \ref{div1}, we get that if $K\in\mathcal{RDG}$ and $M$ is a strongly $(H,1)$-perfect MV-algebra, then $M$ can be embedded in a pseudo MV-algebra with square root.
\end{rmk}

\begin{thm}\label{div4}
Let $M=\Gamma(G,u)$ be such that $G\in \mathcal{RDG}$ is a linearly ordered $\ell$-group with strong unit $u$. The following statements hold:
\begin{itemize}[nolistsep]
\item[{\rm (i)}] $G$ can be embedded into a two-divisible linearly ordered group $A$ with a strong unit. In addition, if $u\in \C(G)$, then
$u\in \C(A)$.
\item[{\rm (ii)}] If $M$ is symmetric, then $M$ can be embedded in a linearly ordered pseudo MV-algebra with strict square root.
\end{itemize}
\end{thm}

\begin{proof}
(i) Let $H$ be a representable two-divisible $\ell$-subgroup containing $G$ as an $\ell$-subgroup.
Then it enjoys unique extraction of roots. Set
$G_0:=G$, $G_1:=G_0/2+G_0/2$ and $G_{n+1}=G_n/2+G_n/2$, where $X/2:=\{x/2\mid  x\in X\}$ for each $X\s H$.
By the assumption, $G_n$'s are well-defined.
We have
\begin{equation*}
G=G_0\s G_1\s G_2\s \cdots\s G_n\s G_{n+1}\s\cdots.
\end{equation*}
Set $A:=\bigcup_{n\in\mathbb N}G_n$.
For each $x,y\in A$ there exists $n\in \mathbb N$ such that $x,y\in G_n$ and so $x+y\in G_{n+1}\s A$ which means $A$ is a subgroup
of $H$ containing $G$ (note that $G=G_0\s A$). Clearly, $A$ is two-divisible.

(1) Since $H$ is representable, $G/n$ is a chain, where $G/n=\{x/n\mid x\in G\}$.

(2) We claim that for each $a\in A$, there exists $m\in\mathbb N$ such that $|a|\leq mu$. Indeed, for $a\in G_0$, it is clear.
If $a\in G_1$, then $a=x/2+y/2$ for some $x,y\in G_0$ and so $|a|=|x/2+y/2|\leq |x/2|+|y/2|+|x/2|+|y/2|$ (see \cite[Page 2, 1.1.3a]{Anderson}).
By (1), we can assume that $|x/2|\leq |y/2|$. It follows that $|a|\leq 4|y/2|$. If $0\leq y/2$, then $|a|\leq 2y$. Choose $m\in\mathbb N$
such that $y\leq mu$, then we have $|a|\leq 2mu$.
If $y/2\leq 0$, then $|a|\leq 4(-y/2)=-2y$. Take $m\in\mathbb N$ such that $-y\leq mu$, then we have $|a|\leq 2mu$.
Let $1\leq n$ and for each $a\in G_n$ there exists $m\in\mathbb N$ such that $|a|\leq m$. Choose $a\in G_{n+1}$.
Then $a=a_1+a_2$ where $a_1,a_2\in G_n$. By the assumption, there is $m\in\mathbb N$ such that $|a_1|,|a_2|\leq mu$.
It follows that $|a|=|a_1+a_2|\leq |a_1|+|a_2|+|a_1|\leq 3mu$.

(3) The group $A$ is linearly ordered. Since $H$ is representable, $A$ and $G$ are representable. Let $\varphi:A\to \prod_{i\in I}A_i$ be a subdirect embedding of unital $\ell$-groups, where each $A_i$ is linearly ordered. For each $i\in I$, set $B_i:=\pi_i\circ\varphi(G)$. Clearly, $\varphi|_{G}:G\to\prod_{i\in I}B_i$ is a subdirect embedding. Since $G$ is subdirectly irreducible (because $G$ is a chain), there exists $j\in I$ such that $\pi_j\circ\varphi$ is an isomorphism,
consequently, $\pi_j\circ\varphi:G\to A_j$ is one-to-one. We claim that $\pi_j\circ\varphi:A\to A_j$ is an isomorphism.
Let $x\in A$ be such that $\pi_j\circ\varphi(x)=0$. Then there exists $n\in\mathbb N$ such that $x\in G_n$. We use induction on $n$.
Clearly, if $x\in G_0=G$, then we get $x=0$. If $x\in G_1$, then $x=x_1+x_2$ for some $x_1,x_2\in G$.
\begin{eqnarray*}
0=\pi_j\circ\varphi(x)=\pi_j\circ\varphi(x_1+x_2)= \pi_j\circ\varphi(x_1)+\pi_j\circ\varphi(x_2),
\end{eqnarray*}
which implies
$\pi_j\circ\varphi(x_1)=-\pi_j\circ\varphi(x_2)$ and so $\pi_j\circ\varphi(2x_1)=\pi_j\circ\varphi(-2x_2)$.
Since $2x_1,-2x_2\in G$, we have $2x_1=-2x_2$ and so $x_1=-x_2$ which entails $x=0$.
Now, let $1\leq n$ and each $\pi_j\circ\varphi$ be one-to-one on $G_n$. As in the proof for $G_1$,
we can show that $\pi_j\circ\varphi$ is one-to-one on $G_{n+1}$. Hence $\pi_i\circ\varphi:A\to A_i$ is
an isomorphism, and consequently, $A_i$ is a chain.
By (1)--(3), $G$ can be embedded in a linearly ordered two-divisible $\ell$-group $A$ with strong unit.

Now, let $u\in \C(G)$. We will show that $u\in \C(A)$. First, we show that for each $a\in G/2$, we have $a+u=u+a$.
Otherwise, assume that $a+u<u+a$ (because $A$ is a chain), then $a+a+u\leq a+u+a$.
From $u\in \C(G)$ and $2a\in G$, it follows that $u+2a=2a+u\leq a+u+a$ and so $u+a\leq a+u$ which contradicts with the assumption.
So, $u+a=a+u$ for all $a\in G/2$ and consequently $a+u=u+a$ for all $a\in G_1$. We can easily conclude that $a+u=u+a$ for all
$a\in G_1/2$, which entails $u$ commutes with all elements of $G_2$. By a similar way, $u$ commutes with all elements of $A:=\bigcup_{n\in\mathbb N}G_n$.

(ii) By (i), there is an $\ell$-group embedding $f:G\to A$, where $A$ is a linearly ordered two-divisible $\ell$-group with a strong unit $f(u)$.
Then $\Gamma(f):\Gamma(G,u)\to\Gamma(A,f(u))$ is an embedding of pseudo MV-algebras.
By Lemma \ref{div01}, $r:\Gamma(A,f(u))\to \Gamma(A,f(u))$ defined by $r(x)=(x-f(u))/2+f(u)$ is a strict square root.
Therefore, the proof of part (ii) is completed.
\end{proof}

\begin{cor}\label{div05}
Let $M=\Gamma(G,u)$ be a symmetric representable pseudo MV-algebra such that $G\in \mathcal{RDG}$. Then $M$ can be embedded in a pseudo MV-algebra with strict square root.
\end{cor}

\begin{proof}
Suppose that $(G,u)$ is a unital $\ell$-group and $f:G\to H$ is an $\ell$-group embedding in a representable two-divisible $\ell$-group $H$. Let $X:=X(M)$ be the set of all proper normal convex $\ell$-subgroup of $H$; we can assume that $G\subseteq H$. Then
$P\cap G$ ($P\cap M$) is a normal convex $\ell$-subgroup of $G$ (normal ideal of the pseudo MV-algebra $M$) for every $P\in X$.
Clearly, $\bigcap_{P\in X}(P\cap M)=\{0\}$.
Consider the $\ell$-group embedding $f_P:M\to \prod_{P\in X}M/(P\cap M)$.
For each $P\in X$, $g_P:G/(P\cap G)\to H/P$ is an embedding, $H/P$ is two-divisible and representable (it is a chain).
In addition, for each $P\in X$, $G/(P\cap G)$ is a chain, hence by Theorem \ref{div4}, $M/(P\cap M)$ can be embedded in a
pseudo MV-algebra $N_P$ with the strict square root $r_P$ (note that $\Gamma(G/(P\cap G),u/(P\cap G))=M/(P\cap M)$). Let $g_P:M/(P\cap M)\to N_P$ be the mentioned embedding for all $P\in X$.
Consider the pseudo MV-algebra $N:=\prod_{P\in X}N_P$ with the strict square root $r:N\to N$ defined by
$r((x_P)_{P\in X})=(r_P(x_P))_{P\in X}$.
Then the mapping $M\xrightarrow{\ \ \alpha\ \ }\prod_{P\in X}M/(P\cap M)\xrightarrow{(g_P)_{P\in X}}\prod_{P\in Y}N_P$ is an
embedding of pseudo MV-algebras.
\end{proof}

It was shown that if $M=\Gamma(G,u)$ is a symmetric pseudo MV-algebra with $G\in\mathcal{RDG}$, then $M$ can be embedded into a pseudo MV-algebra with strict square root. In other words, $M$ has a ``cover" with strict square root. Now, we find the least such ``cover" for $M$.

\begin{defn}\label{div6}
Let $M$ be a pseudo MV-algebra. A {\it strict square root closure} for $M$ is a triple $(C,r,i)$, where $C$ is a pseudo MV-algebra, $r:C\to C$ is a strict square root, and $i:M\to C$ is an embedding of pseudo MV-algebras such that if $N$ is another pseudo MV-algebra with a strict square root $s$
and $j:M\to N$ is an injective homomorphism of pseudo MV-algebras, then there exists a unique injective homomorphism $\alpha:C\to N$ satisfying the following conditions:
\begin{itemize}[nolistsep]
\item[{\rm (C1)}] $\alpha\circ i= j$;
\item[{\rm (C2)}] $\alpha(r(x))=s(\alpha(x))$ for all $x\in M$.
\end{itemize}
\end{defn}

In the last definition, condition (C2) implies that $\alpha$ is a homomorphism of pseudo MV-algebras with square root.
In addition, if $(C,r,i)$ and $(C',r',i')$ are square root closures of $M$, then $C\cong C'$. Hence, $C$ will be denoted by $\mathbf C(M)$, which is determined uniquely up to isomorphism. Sometimes, we need that $M,C,N$ belong to a special class $\mathcal V$ of pseudo MV-algebras, e.g. the class of linear pseudo MV-algebras, etc. If a pseudo MV-algebra $M$ has a strict square root closure, $M$ has to be symmetric.

\begin{lem}\label{div07}
Let $(M_1;\oplus,^-,^\sim,0,1)$ and $(M_2;\oplus,^-,^\sim,0,1)$ be pseudo MV-algebras such that $\mathbf C(M_1)$ and $\mathbf C(M_2)$ exist. The following statements hold:
\begin{itemize}[nolistsep]
\item[{\rm (i)}] $\mathbf C(M_1)\cong \mathbf C(\mathbf C(M_1))$.
\item[{\rm (ii)}] If $f:M\to M_2$ is a one-to-one homomorphism of pseudo MV-algebras and $\mathbf C(M_1)$ exists, then $\mathbf C(M_1)$ can be embedded into $\mathbf C(M_2)$.
\end{itemize}
\end{lem}

\begin{proof}
(i) Let $i_1:M\to \mathbf C(M_1)$ and $i_2:\mathbf C(M_1)\to \mathbf C(\mathbf C(M_1))$ be the embeddings  in Definition \ref{div6} represented in the Figure \ref{fig1}.
\setlength{\unitlength}{1mm}
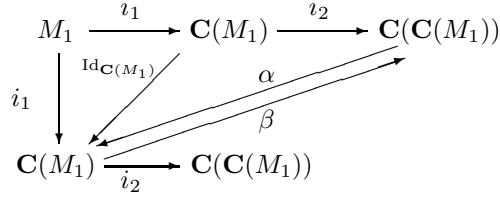
\begin{figure}[!ht]
\begin{center}
\begin{picture}(80,20)
\put(12,17){\vector(2,0){12}}
\put(37,17){\vector(2,0){12}}
\put(8,14){\vector(0,-1){12}}
\put(24,14){\vector(-1,-1){12}}
\put(5,-2){\makebox(4,2){{ $\mathbf C(M_1)$}}}
\put(31,-2){\makebox(4,2){{ $\mathbf C(\mathbf C(M_1))$}}}
\put(5,16){\makebox(4,2){{ $M_1$}}}
\put(28,16){\makebox(4,2){{ $\mathbf C(M_1)$}}}
\put(56,16){\makebox(4,2){{ $\mathbf C(\mathbf C(M_1))$}}}
\put(53,15){\vector(-3,-1){40}}
\put(14,0){\vector(3,1){40}}
\put(15,19){\makebox(4,2){{ $i_1$}}}
\put(15,-4){\makebox(4,2){{ $i_2$}}}
\put(14,-1){\vector(2,0){10}}
\put(40,19){\makebox(4,2){{ $i_2$}}}
\put(1,7){\makebox(4,2){{$i_1$}}}
\put(14,11){\makebox(4,2){{\tiny{$\id_{\mathbf C(M_1)}$}}}}
\put(33,10){\makebox(4,2){{ $\alpha$}}}
\put(33,4){\makebox(4,2){{ $\beta$}}}
\end{picture}
\caption{\label{fig1} The diagram of strict square root closures}
\end{center}
\end{figure} \\
There exist injective homomorphisms $\alpha:\mathbf C(\mathbf C(M_1))\to \mathbf C(M_1)$ and $\beta:\mathbf C(M_1)\to \mathbf C(\mathbf C(M_1))$ such that $\alpha\circ i_2=\id_{\mathbf C(M_1)}$ and
$\beta\circ i_1=i_2\circ i_1$. It follows that $\alpha\circ\beta:\mathbf C(M_1)\to \mathbf C(M_1)$ commutes the diagram, that is, $\alpha\circ\beta\circ i_1=i_1$.
On the other hand, $\id_{\mathbf C(M_1)}\circ i_1=i_1$, so $\alpha\circ\beta=\id_{\mathbf C(M_1)}$ which means $\alpha$ is onto. Therefore, $\alpha:\mathbf C(\mathbf C(M_1))\to \mathbf C(M)$ is an isomorphism.

(ii) Let $i_1:M_1\to \mathbf C(M_1)$ and $i_2:M_2\to \mathbf C(M_2)$ be the embeddings in Definition \ref{div6}.
Since $i_2\circ f:M_1\to \mathbf C(M_2)$ is an embedding, there exists a map $\alpha:\mathbf C(M_1)\to \mathbf C(M_2)$ such that
$\alpha\circ i_1=i_2\circ f$. Since $\alpha$ is one-to-one, $\mathbf C(M_1)$ can be embedded into $\mathbf C(M_2)$.
\end{proof}

\begin{thm}\label{div7}
Each MV-algebra has a strict square root closure.
\end{thm}

\begin{proof}
Let $M=\Gamma(G,u)$ be an MV-algebra and $i:G\to D$ be an $\ell$-group embedding, where $D$ is an Abelian divisible unital $\ell$-group.
We can assume that $i$ is an inclusion map. Define $D_0=G$ and $D_n=D_{n-1}/2$ for all $n\in\mathbb N$.
Note that since $D$ is representable, $D$ enjoys unique extraction of roots.
Assume that $H:=\bigcup_{i\in\mathbb N} D_i$. Clearly, $G\s H$ and $D_{n-1}\s D_n$ for all $n\in\mathbb N$. We assert that $\Gamma(H,u)$ is a strict square closure for $M$.

(1) Let $x,y\in G$. Then $x/2+y/2=(x+y)/2$ (because $D$ is commutative) and $D_1$ is closed under $+$. Also,
$2(x/2\wedge y/2)\leq x,y$, so $2(x/2\wedge y/2)\leq x\wedge y$, that is $x/2\wedge y/2\leq (x\wedge y)/2$.
On the other hand, $x\wedge y\leq x,y$ implies that $(x\wedge y)/2\leq x/2,y/2$, so that $(x\wedge y)/2\leq x/2\wedge y/2$.
It follows that, $x/2\wedge y/2=(x\wedge y)/2$. Similarly, $x/2\vee y/2=(x\vee y)/2$. Thus, $D_1$ is closed under $\vee$ and $\wedge$.
Similarly, we can prove that $D_n$ is closed under $+$, $\vee$, and $\wedge$. Therefore, $H$ is an $\ell$-subgroup of $D$ which is
two-divisible. By Lemma \ref{div01}, $r:\Gamma(H,u)\to \Gamma(H,u)$ defined by $r(x)=(x+u)/2$ is a strict square root on $\Gamma(H,u)$
and $i:\Gamma(G,u)\to \Gamma(H,u)$ is an embedding of MV-algebras.

(2) From (1), we conclude that for each $x\in H$, there exists $n\in\mathbb N\cup \{0\}$ such that $2^{n}x\in G$. Let $n_x$ be the least element $n$ of $\mathbb N\cup \{0\}$ such that $2^{n}x\in G$.
For example, if $x\in G$, then $n_x=0$, if $x\in D_1\setminus D_0$, then $n_x=1$, and if $x\in D_n\setminus D_{n-1}$, then $n_x=n$.

(3) Let $N$ be an MV-algebra with a strict square root $s$, and $f:M\to N$ be a homomorphism of MV-algebras.
By Theorem 4.4, $(A,w):=\Gamma^{-1}(N)$ is a two-divisible unital $\ell$-group with strong unit $w$.
Consider the $\ell$-group homomorphism
$F:=\Gamma(f):G\to A$. Define $\alpha:H\to A$ by $\alpha(x)=F(2^{m}x)/2^{m}$, where $2^mx\in G$. Let $x\in H$ and $m\in\mathbb N$ such that
$2^mx\in G$. Then $m=n_xp+q$ where $p,q\in \mathbb N\cup \{0\}$ and $0\leq q<n_x$. By definition of $n_x$, we have $q=0$
(otherwise, $2^q x\in G$ which is absurd). That is, $n_x|m$ and $\alpha(x)=F(2^{m}x)/2^{m}=F(2^{n_xp}x)/2^{n_xp}=
F(2^{n_x}x)/2^{n_x}$, whence $F$ is well-defined. Also, $\alpha(0)=0$, $\alpha(u)=w$ and $\alpha(i(g))=F(x)$ for all $g\in G$.

(4) Let $x,y\in H$ be given. Then there exist $m,n\in\mathbb N$ such that $2^mx,2^ny\in G$ and so
$2^{mn}x,2^{mn}y\in G$. Thus $\alpha(x+y)=F(2^{mn}(x+y))/2^{mn}=F(2^{mn}x)/2^{mn}+F(2^{mn}y)/2^{mn}=F(2^{m}x)/2^{m}+F(2^{n}y)/2^{n}
=\alpha(x)+\alpha(y)$ (since $F$ is a homomorphism and $G$, $H$ and $A$ enjoy unique extraction of roots).

(5) If $x\leq y$, then $2^{mn}x\leq 2^{mn}y$ and so $F(2^{mn}x)\leq F(2^{mn}y)$ consequently,
$\alpha(x)=F(2^{mn}x)/2^{mn}\leq F(2^{mn}y)/2^{mn}=\alpha(y)$. Hence, $\alpha:H\to A$ is ordered-preserving.
Moreover, if $M$, or equivalently $G$, is a chain, then $D$ and $H$ are also chains. So, $\alpha$ is an $\ell$-group homomorphism.

(6) Similarly to (1), we can show that $x/2^n\wedge y/2^n=(x\wedge y)/2^n$ for all $x,y\in G$, and for $2^nx,2^ny\in G$, we have
$x\wedge y=(2^nx\wedge 2^ny)/2^n$. Thus, $2^nx\wedge 2^ny=2^n(x\wedge y)$. It follows that
\begin{eqnarray*}
\alpha(x)\wedge \alpha(y)&=& \dfrac{F(2^{mn}x)}{2^{mn}}\wedge \dfrac{F(2^{mn}y)}{2^{mn}}=\dfrac{F(2^{mn}x)\wedge F(2^{mn}y)}{2^{mn}}\\
&=& \dfrac{F(2^{mn}x\wedge 2^{mn}y)}{2^{mn}}=\dfrac{F(2^{mn}(x\wedge y))}{2^{mn}}=\alpha(x\wedge y).
\end{eqnarray*}

(7) Analogously to (6), $\alpha$ preserves $\vee$.

(8) If $x\in D_n$, then  $\alpha(r(x))=\alpha((x+u)/2)=F(2^{n+1}(x+u)/2)/2^{n+1}=(F(2^{n}x)+F(2^{n}u))/2^{n+1}=F(2^{n}x)/2^{n+1}+w/2=(\alpha(x)+w)/2$.
Moreover, by Theorem 4.4, $A$ is two-divisible, so Lemma \ref{div01} implies that $k:N\to N$, defined by $k(x)=(x+w)/2$, $x\in N$, is a square root on $N$, consequently $k=s$ (because the square root on a pseudo MV-algebra is unique). Hence,
$s(\alpha(x))=(\alpha(x)+w)/2=\alpha(r(x))$.

Consequently, $\alpha$ is a one-to-one homomorphism of pseudo MV-algebras with square roots. Summing up all the above cases, we conclude that $\Gamma(H,u)$ is a strict square closure for the MV-algebra $M$.
\end{proof}

The following notion will be used in the following statement. If $X$ is a subset of an $\ell$-group $G$, then $\sum_{i=1}^n X:=\{x_1+\cdots+x_n\mid x_1,\ldots,x_n\in X\}$.

\begin{cor}\label{co:RDP}
Let $M=\Gamma(G,u)$ be an MV-algebra, $\mathbf C(M)$ be its strict square root closure, and assume $M \subseteq \mathbf C(M)$. For each $x\in \mathbf C(M)$, there exists an integer $n\in\mathbb N$ such that
$2^nx\in \sum_{i=1}^{2^n}M$. That is, $\mathbf C(M)=\bigcup_{n=0}^\infty (\sum_{i=1}^{2^n}M/2^{n})$.
\end{cor}

\begin{proof}
Consider the notations in the proof of Theorem \ref{div7}.
For each $n\in\mathbb N$ set $M_n:=\sum_{i=1}^{2^n}M/2^{n}$.
According to Lemma \ref{div07}, $\mathbf C(M)$ is an MV-algebra with strict square root generated by $M$.
By Theorem \ref{div7}, $\mathbf C(M)=\Gamma(H,u)$, where $H=\bigcup_{i\in\mathbb N}D_i$. For each $y\in M$, $y/2=u-(u-y+u)/2=r(y')'\in \mathbf C(M)$, where $r(x)=(x+u)/2$, $x\in \mathbf C(M)$ is a (unique) strict square root on $\mathbf C(M)$. It entails that
$M/2\s \mathbf C(M)$. Since $x,y\leq u/2$ for all $x,y\in M/2$, we have $x+y=(x+y)\wedge u=x\oplus y\in \mathbf C(M)$. That is, $M_1\s \mathbf C(M)$. In a similar way, we can show that
$M_n\s \mathbf C(M)$.

Now, let  $x\in \mathbf C(M)$. If $x\in D_0=G$, then $x\in \Gamma(G,u)=M/2^0$. If $x\in D_1$, then $x=g/2\leq u$ for some $g\in D_0$, so $g\leq 2u$.
Due to the Riesz Interpolation Property (RIP), there exist $g_1,g_2\in \Gamma(G,u)$ such that $g=g_1+g_2$, whence $x=(g_1+g_2)/2=g_1/2+ g_2/2\in M/2+M/2$. On the other hand, $g_1/2+g_2/2=g_1/2\oplus g_2/2\in M/2\oplus M/2$, where $\oplus$ is the binary operation on the MV-algebra $\Gamma(H,u)$, since
$g_1/2,g_2/2\leq u/2$.

Let $0\leq n$. If $x\in D_{n}$, then $x=g/2^{n}$ for some $g\in G$, so $2^{n}x=g\leq 2^{n}u$.
Due to (RIP), there exist positive elements $g_1,\ldots, g_{2^{n}}\leq u$ of $G$ (that is,  $g_1,\ldots, g_{2^{n}} \in M$) such that $g=\sum_{i=1}^{2^{n}}g_i$,
whence $x=g/2^{n}=(\sum_{i=1}^{2^{n}}g_i)/2^n=\sum_{i=1}^{2^{n}}(g_i/2^n)\in \sum_{i=1}^{2^{n}}M/2^n$.
Therefore, $\mathbf C(M)\s\bigcup_{n=0}^\infty (\sum_{i=1}^{2^n}M/2^{n})$, and so, we have $\mathbf C(M)=\bigcup_{n=0}^\infty (\sum_{i=1}^{2^n}M/2^{n})$.
\end{proof}

The following example shows that there are non-commutative linearly ordered symmetric pseudo MV-algebras without strict square roots that can be embedded into a linearly ordered symmetric pseudo MV-algebra with strict square root with some kind of minimality of the embedding.

\begin{exm}
Let $G=\mathbb Z\lex \mathbb Z \lex \mathbb Z\lex \mathbb Z$, $u=(1,0,0,0)$, $H=\mathbb D\lex \mathbb D \lex \mathbb D\lex \mathbb D$, where $\mathbb D$ is the group of dyadic numbers in $\mathbb R$, and let the group operations on $G$ and $H$ be defined by (4.2). The groups $G$ and $H$ are linearly ordered, $u\in \C(G)$, $G$ is not two-divisible whereas $H$ is two-divisible, and $G$ is a subgroup of $H$.
The pseudo MV-algebras $M=\Gamma(G,u)$ and $C=\Gamma(H,u)$ are symmetric and linearly ordered, $M$ is not two-divisible, and $C$ is two-divisible. We assert that $\Gamma(H,u)$ is the least two-divisible pseudo MV-algebra containing $\Gamma(G,u)$ and contained in $\Gamma(H,u)$, moreover $\Gamma(H,u)$ is with strict square root.

Indeed, first, we show that $H$ is the least subgroup of $H$ that is two-divisible and contains $G$. We have $\mathbb D=\{m/2^n\mid m\in \mathbb Z, n\ge 0\}$ and $\mathbb D$ is two-divisible. Let $H_1$ be a two-divisible subgroup of $H$ containing $G$. Then $\mathbb D\times\{0\}\times\{0\}\times \{0\},  \{0\}\times\mathbb D\times\{0\}\times\{0\}, \{0\}\times \{0\}\times \mathbb D\times\{0\}, \{0\}\times \{0\}\times\{0\}\times \mathbb D \subseteq H_1$. Take $(a,b,c,d)\in H$. Then $(a,b,c,d)= (a,0,0,0)+(0,b,0,0)+ (0,0,c,0)+(0,0,0,d-bc)\in H_1$, so that $H_1=H$.
\end{exm}

In Corollary \ref{div05}, we showed that if $M=\Gamma(G,u)$ is a pseudo MV-algebra such that $G\in\mathcal{RDG}$, then $M$ can be embedded into a pseudo MV-algebra with strict square root. There is another way for embedding $M$ into a pseudo MV-algebra with square root. We will prove that $M$ can be embedded in a pseudo MV-algebra with square root, not necessarily strict.

\begin{prop}\label{10.1}
Every symmetric pseudo MV-algebra $M=\Gamma(G,u)$, where $G \in \mathcal{RDG}$, can be embedded into a symmetric pseudo MV-algebra with square root, not necessarily strict.
\end{prop}

\begin{proof}
Let $M=\Gamma(G,u)$ be a symmetric pseudo MV-algebra such that $G\in\mathcal{RDG}$.
There exists an $\ell$-group embedding $h:G\to H$, where $H$
is two-divisible and representable. Let $Y$ be the set of all normal prime ideals of $H$.
By \cite[Thm 4.1.1]{Anderson}, the $\ell$-group $G$ is also representable.
Set $X=\{h^{-1}(Q)\mid Q\in Y\}$. Then each element of $X$ is a normal convex $\ell$-subgroup of $G$ and $\bigcap_{P\in X}P=\{0\}$.
It follows that $f:G\to \prod_{P\in X}G/P$ defined by $f(g)=(g/P)_{P\in X}$ is a one-to-one $\ell$-group homomorphism.
Choose $f^{-1}(Q)=P\in X$. We know that $\{P\cap M\mid P\in X\}$ is a set of normal prime ideals of $M$, and $M$ is a subdirect product of the family
$\{M_P=M/(P\cap M)\mid P\in X\}$ of linearly ordered pseudo MV-algebras.

Every $M_P$ is one of the following three cases.

(i) If $|M_P|=1$, then we can remove $M_P$.

(ii) If $|M_P|=2$, then $M_P=\{0/(P\cap M),1/(P\cap M)\}$. We set $r_P=\id_{M_P}$ and $K_P=\Gamma(\mathbb Z,1)$.
Let $g_P:M_P\to \Gamma(\mathbb Z,1)$ be the only isomorphism between Boolean algebras with two elements.

(iii) Suppose that $3\leq|M_P|$. We have $\Gamma(G/P,u/P)=M_P$. Since the mapping $h_P:G/P\to H/Q$ sending
$g/P$ to $h(g)/Q$ is an $\ell$-group embedding and $H/Q$ is two-divisible, so $G/P$ is a linearly ordered group belonging to $\mathcal{RDG}$.
By the proof of Theorem \ref{div4}(i), $A_P:=\bigcup G_{P,n}$ is a two-divisible $\ell$-group, where $G_{P,0}=G_P$ and $G_{P,n+1}=G_{P,n}/2+G_{P,n}/2$. Since $u\in \C(G)$, by that theorem $(u/P)/2\in \C(A_P)$. Lemma \ref{div01} implies that
there exists a strict square root $r_P$ on $K_P:=\Gamma(A_P,u/P)$. Set $g_P=\Gamma(h_P)$.

Now, define the pseudo MV-algebra $M_0:=\prod_{P\in X}K_P$, where $K_P$ is either $\Gamma(\mathbb Z,1)$, if $|M_P|=2$ or $\Gamma(A_P,u/P)$, if $|M_P|>2$, and let $r:M_0\to M_0$ be the square root defined by $r((x_P)_{P\in X})=(r_P(x_P))_{P\in X}$. If all $|M_P|>2$, $r$ is a strict square root on $M_0$, otherwise $r$ is not strict.

Consider the embedding $\varphi:M\xrightarrow{f_M}\prod_{P\in X}M/(P\cap M)\xrightarrow{(g_P)_{P\in X}}\prod_{P\in X}K_P=M_0$, where $f_M$ is the restriction of $f$ onto $M$. Then $M$ can be embedded into a symmetric pseudo MV-algebra $M_0$ with square root, not necessarily strict.
\end{proof}

A normal ideal $I$ of a pseudo MV-algebra $(M;\oplus,^-,^\sim,0,1)$ is called  (i) a {\em strict square ideal} if $M/I$ is a pseudo MV-algebra with a strict square root, and (ii) a {\it Boolean ideal}, if $M/I$ is a Boolean algebra (equivalently,  $x\wedge x^-\in I$ for all $x\in M$).

The following lemma will be useful for the proof of Proposition \ref{10.5}.

\begin{lem}\label{strictlem}
Let $I$ be a normal ideal of a pseudo MV-algebra $(M;\oplus,^-,^\sim,0,1)$ with a square root $r$ and let $w=r(0)^-\odot r(0)^-$.
\begin{itemize}
\item[{\rm (i)}]  $I$ is a strict square ideal \iff $w\in I$.
\item[{\rm (ii)}] If $J$ is a normal ideal of $M$ containing a strict square ideal $I$, then $J$ is a strict square ideal, too.

The intersection of each family of strict square ideals of $M$ is a strict square ideal.
\item[{\rm (iii)}] The interval $[0,w]$ is the least square ideal of $M$.
\end{itemize}
\end{lem}

\begin{proof}
(i) Let $r$ be a square root on $M$. Due to \cite[Prop 3.9]{DvZa3}, $r_I:M/I\to M/I$ defined by $r_I(x/I)=r(x)/I$ is a square root on $M/I$.
\begin{eqnarray}
\mbox{ $I$ is strict } \Leftrightarrow r_I(0/I)^-\odot r_I(0/I)^-=0/I \Leftrightarrow (r(0)^-\odot r(0)^-)/I=0/I \Leftrightarrow w\in I.
\end{eqnarray}
(ii) The proof is a straight consequence of (i).

(iii) By \cite[Prop 3.3(11)]{DvZa3}, $w$ is a Boolean element and by \cite[Thm 4.3]{DvZa3}, $M\cong [0,w]\times [0,w^-]$
(the mapping $f:M\to [0,w]\times[0,w^-]$ defined by $f(x)=(x\wedge w,x\wedge w^-)$ is the isomorphism).
Therefore, $[0,w]$ is a normal ideal of $M$, and $M/[0,w]$ is isomorphic to $[0,w^-]$ which is a strict pseudo MV-algebra. It entails that $[0,w]$ is a strict square ideal.
On the other hand, if $J$ is a strict square ideal of $M$, then by (ii), $w\in J$ and so $[0,w]\s J$. Therefore,
$[0,w]$ is the least square ideal of $M$.
\end{proof}

For each non-degenerate pseudo MV-algebra $M$,
$X(M)$ denotes the set of all proper normal prime ideals of $M$, $X(M)_1$ is the set of all Boolean normal prime ideals of $M$ and $X(M)_2=X(M)\setminus X(M)_1$. In addition, we define $I_1:=\bigcap X(M)_1$ and $I_2:=\bigcap X(M)_2$.
If there is no ambiguity, we will denote them simply by $X$, $X_1$, and $X_2$, respectively.

\begin{prop}\label{10.5}
Let $M=(M;\oplus,^-,^\sim,0,1)$ be a representable pseudo MV-algebra with a square root  $r$. The following statements hold:
\begin{itemize}[nolistsep]
\item[{\rm (i)}] If $r$ is strict, then $P$ is not a Boolean ideal, for each $P\in X(M)\setminus\{M\}$.
\item[{\rm (ii)}] If $M\cong M_1\times M_2$ is the decomposition of $M$ that was mentioned in  \cite[Thm 4.3]{DvZa3}, where $M_1$ is a Boolean algebra and $M_2$ is a strict MV-algebra, then $M_i$ is isomorphic to a subalgebra of $\prod_{P\in X_i}(M/P)$, for $i=1,2$.
\item[{\rm (iii)}] The ideal $M_2=\bigcap_{P\in X_1(M)}P$ is the least Boolean ideal and $M_1=\bigcap_{P\in X_2(M)}P$ is the least strict square ideal of $M$.
In addition, $M$ is strict \iff $\bigcap_{P\in X_2}P=\{0\}$.
\end{itemize}
\end{prop}

\begin{proof}
(i) Choose $P\in X(M)\setminus\{M\}$. Due to the onto homomorphism $\pi_P:M\to M/P$ and \cite[Prop 3.9]{DvZa3}, $r_P:M/P\to M/P$
defined by $r_P(x/P)=r(x)/P$ is a square root, which is clearly strict, since $r$ is strict.
Hence $r_P(0/P)=r_P(0/P)^-$.
If $M/P$ is a Boolean algebra, then it has only two elements, that is, $M/P=\{0/P,1/P\}$ and
by \cite[Thm 3.8]{DvZa3}, $r_p(0/P)=0/P$, so $0/P=(0/P)^-=1/P$ which is absurd.
Therefore, $M/P$ is not a Boolean algebra.

(ii) According to \cite[Thm 4.4]{DvZa3}, $M\cong M_1\times M_2$, where $M_1=[0,w]$ is a Boolean algebra, $M_2=[0,w^-]$ is a strict MV-algebra, and $w=r(0)^-\odot r(0)^-$.
Set $E_1:=\{(x_P/P)_{P\in X(M)}\in \prod_{P\in X(M)}M/P\mid x_P=0,~\forall P\in X_2(M)\}$ and $E_2=\{(x_P/P)_{P\in X(M)}\in \prod_{P\in X(M)}M/P\mid x_P=0,~\forall P\in X_1(M)\}$.

Let $a=(a_P/P)_{P\in X(M)}$, where $a_P=1$ if $P\in X_1(M)$ and $a_P=0$ if $P\in X_2$. We define a mapping $g:\prod_{P\in X(M)}M/P\to E_1\times E_2$ by $g(x)=(x\wedge a,x\wedge a')$ for all $x\in \prod_{P\in X(M)}M/P$. Then $g$ is an isomorphism. Let $f:M\to \prod_{P\in X(M)}M/P$ be the natural subdirect embedding.
Consider the embedding $\varphi:=g\circ f$.

Case 1. If $P\in X_2(M)$, then $M/P$ is not a Boolean algebra, so $r_P:M/P\to M/P$ defined by $r(x/P)=r(x)/P$ is strict, by \cite[Cor 4.5]{DvZa3}, and in other words, $P$ is a strict square ideal. We have $(r(0)^-\odot r(0)^-)/P=r_P(0/P)^-\odot r_P(0/P)^-=0/P$, that is, $w\in P$ and $w^-/P=1/P$.

Case 2. If $P\in X_1(M)$, then $M/P$ is a Boolean algebra, so $r_P(0/P)=0/P$. Thus $(r(0)^-\odot r(0)^-)/P=r_P(0/P)'\odot r_P(0/P)=1/P$, so that $w/P=1/P$ and
$w'/P=0/P$.

Hence, $f(w)\in E_1$ and $f(w^-)\in E_2$ yielding $f(M_1)=f([0,w])\s E_1$ and $f(M_2)=f([0,w^-])\s E_2$.
The element $f(w)$ is a top element of $E_1$, and $f(w^-)$ is a top element of $E_2$. Therefore, $M_1$ is isomorphic to a subalgebra of $E_1$, and $M_2$ is isomorphic to a subalgebra of $E_2$.
In addition, $M_1=\{0\}$ \iff $w=0$ \iff $E_1=\{0\}$. Similarly, $M_2=\{0\}$ \iff $w^-=0$ \iff $E_2=\{0\}$.

(iii) Clearly, $I=\bigcap_{P\in X_1(M)}P$ is a Boolean ideal of $M$ (it is normal, too). If $J$ is a Boolean ideal of $M$, then
$J=\bigcap\{P\mid J\s P\}$, since $M/J$ is representable. Each element $P$ of this set is a Boolean ideal of $M$ (because it contains $J$).
It follows that $I\s J$. Since $w^-\in\B(M)$, the set $M_2=[0,w^-]$ is an ideal of $M$ which is normal (since $M\cong M_1\times M_2$ and $\{0\}\times M_2$ is a normal ideal of
$M_1\times M_2$). Also, $M/M_2\cong M_1$ is a Boolean algebra, so $I\s M_2$. On the other hand, by the proof of part (ii),
for each $P\in X_1(M)$, we have $w^-\in P$, whence $M=[0,w']\s I$. Therefore, $M_2=\bigcap_{P\in X_1(M)}P$.

Let $K=\bigcap\{P\mid P\in X_2(M)\}$. According to (ii), each $P\in X_2(M)$ is a strict square ideal, so by Lemma \ref{strictlem},
$K$ is a strict square ideal of $M$. Lemma \ref{strictlem} implies that $[0,w]\s K$.

On the other hand, $[0,w]=\bigcap\{P\in X(M)\mid [0,w]\s P\}=\bigcap\{P\in X(M)\mid w\in P\}$, since $M$ is representable, so by Lemma \ref{strictlem}, each element of the set $\{P\in X(M)\mid [0,w]\s P\}$ is a strict square ideal.
Consequently, $M/P$ is not a Boolean algebra. Hence $\{P\in X(M)\mid [0,w]\s P\}\s X_2(M)$ and $K\s[0,w]$.
Therefore,  $K=[0,w]$ is the least strict square ideal of $M$; see Lemma \ref{strictlem}(iii).

The proof of the other parts is straightforward.
\end{proof}

In the following theorem, we characterize representable pseudo MV-algebra with square root using normal prime ideals.

\begin{thm}
Let $(M;\oplus,^-,^\sim,0,1)$ be a representable pseudo MV-algebra with a square root $r$. Then  $M$ is strict \iff $M$ does not have any proper Boolean normal prime ideal.

Consequently, if $M$ has a proper  Boolean normal prime ideal, then $M$ is not strict.
\end{thm}

\begin{proof}
Suppose that $M$ is strict. If $P$ is a Boolean normal prime ideal of $M$, then $M/P$ is a Boolean algebra.
By \cite[Prop 3.9]{DvZa3}, $r_P(x)=r(x)/P$ is a square root on $M/P$. Since $r$ is strict, $r(0/P)=r(0)/P=r(0)^-/P=(r(0)/P)^-$ and the Booleanicity of $M/P$ entails $r(0/P)=0/P$. It follows that $0/P=1/P$, which means $P=M$. Conversely, assume that $M$ has no proper Boolean normal prime ideal.
If $M$ is degenerate, then clearly, the statement holds. Let $M$ be non-degenerate. The pseudo MV-algebra $M$ is not a Boolean algebra. Otherwise, each prime ideal of $M$ is a Boolean normal prime ideal. So, $X(M)=X(M)_2$.
Set $w=r(0)^-\odot r(0)^-$. By Lemma \ref{strictlem}(iii), $w\in\bigcap\{P\colon P\in X(M)_2\}=\bigcap\{P\colon P\in X(M)\}=\{0\}$. Therefore, $r(0)=r(0)^-$ and $r$ is a strict square root on $M$.
\end{proof}

Finally, we present a criterion showing when an MV-algebra $N$ with strict square root is a strict square root closure of an MV-algebra. It will be used in Example \ref{ex:D(M)}.

\begin{thm}\label{th:crit}
Let $M=\Gamma(G,u)$ be an MV-algebra which is a subalgebra of an MV-algebra $N=\Gamma(H,u)$ with strict square root.
The following statements are equivalent:
\begin{itemize}
\item[{\rm (i)}] $N\cong \mathbf C(M)$.
\item[{\rm (ii)}] For each $h\in H$, there is an integer $n$ such that $2^nh \in G$.
\item[{\rm (iii)}] For each $x\in N$, there are an integer $n$ and elements $x_1,\ldots,x_{2^n}\in M$ such that $x= x_1/2^n+\cdots+x_{2^n}/2^n$.
\end{itemize}
\end{thm}

\begin{proof}
Due to Mundici's representation of MV-algebras, we can assume that $(G,u)$ is a unital $\ell$-subgroup of the unital $\ell$-group $(H,u)$. Moreover, Theorem 4.4 says that $H$ is two-divisible, and $s(y)=(y+1)/2$, $y\in N$, is a unique strict square root on $N$. Let $r$ be a unique strict square root on $\mathbf C(M)$.

(i) $\Rightarrow$ (ii). If $\mathbf C(M)\cong N$, then statement (ii) follows from the proof (2) of Theorem \ref{div7}.

(ii) $\Rightarrow$ (iii). If $x\in N$, then by (ii), there is an integer $n\ge 0$ such that $2^nx \in G$. Then $2^nx\le 2^nu$, so by (RIP), there are $x_1,\ldots,x_{2^n}\in M$ such that $2^nx= x_1+\cdots+x_{2^n}$, so that  $x = x_1/2^n+\cdots+x_{2^n}/2^n$.

(iii) $\Rightarrow$ (ii). If $h\in N$, there is $n$ and $h_1,\ldots, h_{2^n}\in M$ such that $2^nh=h_1+\cdots+h_{2^n}$ and $2^nh\in G$. Let $h\in H^+$. Due to (RIP), there are elements $h_1,\ldots,h_k\in N$ such that $h=h_1+\cdots +h_k$. By (iii), every $h_i$ is of the form $h_i= x_{i1}/2^{n_i}+\cdots+ x_{i2^{n_i}}/2^{n_i}$. We can assume $n_1\le \cdots\le n_k$. Then $x_{ij}/2^{n_i}= 2^{n_k-n_i}x_{ij}/2^{n_k}$ and $2^{n_k-n_i}x_{ij}$ belongs to $G$. Therefore, for $n=n_k$, we have $2^{n}h\in G$. Finally, let $h\in H$. It is of the form $h= h^+-h^-$, where $h^+,h^-\in H^+$. Now, it is easy to see that there is $n\in \mathbb N$ such that $2^nh\in G$.

(ii) $\Rightarrow$ (i). Let $i_C:M\to \mathbf C(M)$ be the embedding from the definition of the strict square root closure and let $j_N:M\to N$ be the inclusion embedding, $j_N(x)=x$, $x\in M$. There is a unique injective homomorphism $\alpha: \mathbf C(M)\to N$ such that $\alpha\circ i_C=j_N$ and $\alpha \circ r= s\circ \alpha$.

Since (ii) and (iii) are equivalent, there are an integer $n$ and elements $x_1,\ldots,x_{2^n}\in M$ such that $x=x_1/2^n+\cdots+x_{2^n}/2^n$. Since every $x_i \in M$, then $i_C(x_i), i_C(x_i)/2^n \in \mathbf C(M)$. Moreover, $i_C(x_1)/2^n+\cdots+ i_C(x_{2^n})= i_C(x_1)/2^n\oplus \cdots\oplus i_C(x_{2^n})\in \mathbf C(M)$. Applying $\alpha$, we get
\begin{align*}
\alpha(i_C(x_1)/2^n)+\cdots+ \alpha(i_C(x_{2^n})/2^n) &=
\alpha(i_C(x_1))/2^n+\cdots+ \alpha(i_C(x_{2^n}))/2^n\\
&= j_N(x_1)/2^n+\cdots + j_N(x_{2^n})/2^n\\
&= x_1/2^n+\cdots + x_{2^n}/2^n=x,
\end{align*}
which proves that $\alpha(\mathbf C(M))=N$, and $\alpha$ is an injective and surjective homomorphism. Consequently $N\cong \mathbf C(M)$.
\end{proof}

\section{Square root closure and strict square root closure}

In this section, we introduce the notion of a square root closure of an MV-algebra, and we compare it with the strict square root.

\begin{defn}\label{n1}
A pseudo MV-algebra $M$ is called {\em Boolean subdirectly irreducible} if  each subdirect embedding $f:M\to \prod_{i\in J}M_i$ can be reduced to
the embedding $f_2:M\to \prod_{i\in J_2}M_i$, where $f_2(x)=(\pi_i\circ f(x))_{i\in J_2}$, and $J_2:=\{i\in J\mid M_i \mbox{ is not a Boolean algebra}\}$.
Equivalently,
\begin{eqnarray}
\label{ne1}  f^{-1}(\{(x_i)_{i\in J}\in\prod_{i\in I}M_i\mid x_i=0,~\forall i\in J_2\})=\Ker f_2=\{0\}.
\end{eqnarray}
\end{defn}

\begin{thm}\label{n3}
Let $M$ be a pseudo MV-algebra. Then $M$ is Boolean subdirectly irreducible  \iff each one-to-one homomorphism $f:M\to N_1\times N_2$ of pseudo MV-algebras, where $N_1$ is a Boolean algebra and $N_2$ is a pseudo MV-algebra, can be reduced to the embedding $\pi_2\circ f:M\to N_2$.
\end{thm}

\begin{proof}
First, assume that $M$ is Boolean subdirectly irreducible. Let $f:M\to N_1\times N_2$ be a one-to-one homomorphism, where $N_1$ and $N_2$ are a Boolean algebra and a pseudo MV-algebra, respectively.
It suffices to show that $\Ker(\pi_2\circ f)=\{0\}$. Clearly, $\Ker(\pi_2\circ f)=f^{-1}(N_1\times \{0\})$. Set
$I:=f^{-1}(N_1\times \{0\})$ and $J=f^{-1}(\{0\}\times N_2)$. We have $I\cap J=\{0\}$, so $g:M\to M/I\times M/J$ defined by
$g(x)=(x/I,x/J)$ is a subdirect embedding, where $M/I$ is isomorphic to a subalgebra of $N_2$ and $M/J$ is isomorphic to a subalgebra of $N_1$.
Since $N_1$ is a Boolean algebra, by the assumption $\pi_2\circ g:M\to M/I$ is an embedding, which implies that $I=\{0\}$.
Therefore, $\pi_2\circ f:M\to N_2$ is an injective homomorphism.

Conversely, let $f:M\to \prod_{i\in I}M_i$ be a subdirect embedding. Set $I_1=\{i\in I\mid M_i \mbox{ is a Boolean algebra}\}$ and $I_2=I\setminus I_1$.
Suppose that $E_j=\prod_{i\in I_j}M_i$ for $j=1,2$. Consider the map $g:M\to E_1\times E_2$ defined by $g(x)=((\pi_i\circ f(x))_{i\in I_1},(\pi_i\circ f(x))_{i\in I_2})$.
The map $g$ is a one-to-one homomorphism of pseudo MV-algebras. Since $E_1$ is a Boolean algebra, by the assumption,
$\{0\}=\Ker(\pi_2\circ g)=g^{-1}(E_1\times\{0\})=\{x\in M\mid \pi_i\circ f(x)=0,~ \forall i\in I_2\}$.

It follows that $f^{-1}(\{(x_i)_{i\in I}\in\prod_{i\in I}M_i\mid x_i=0,~\forall i\in I_2\})=\{0\}$.
Hence, $M$ is Boolean subdirectly irreducible.
\end{proof}

The following theorem shows a relation between Boolean subdirectly irreducible pseudo MV-algebras and strict pseudo MV-algebras.
Hence, it will provide us with many examples of Boolean subdirectly irreducible pseudo MV-algebras.

\begin{thm}\label{n6}
Let $M$ be a pseudo MV-algebra with a square root $r$. Then $M$ is Boolean subdirectly irreducible \iff $r$ is strict.
\end{thm}

\begin{proof}
Let $r$ be a strict square root. If there exists a subdirect embedding $f:M\to B\times A$, where $B$ is a Boolean algebra and $A$ is a pseudo MV-algebra, then the onto homomorphism $\pi_1\circ f:M\to B$ induces a square root $r_1:B\to B$ defined by
$r_1(\pi_1\circ f(x))=\pi_1\circ f(r(x))$ for all $x\in M$. It follows that $r_1(0)=\pi_1\circ f(r(0))=\pi_1\circ f(r(0)^-)=(\pi_1\circ f(r(0)))^-=r_1(0)^-$.
On the other hand, $r_1(0)=0$, since $B$ is a Boolean algebra (see \cite[Thm 3.8]{DvZa3}).
Hence $0=1$ on $B$, and so $B$ has only one element. It follows that $\Ker(\pi_2\circ f)=f^{-1}(B\times \{0\})=\{0\}$.
Therefore, by Theorem \ref{n3}, $M$ is Boolean subdirectly irreducible.

Conversely, let $M$ be Boolean subdirectly irreducible. By \cite[Thm 4.3]{DvZa3}, $M\cong [0,v]\times [0,v^-]$ where $v=r(0)^-\odot r(0)^-$; the mapping $f:M\to [0,v]\times [0,v']$ defined by $f(x)=(x\wedge v, x\wedge v')$ is an isomorphism

In addition, $[0,v]$ is a Boolean algebra and $[0,v^-]$ is a strict pseudo MV-algebra.
By the assumption, $f_2:M\to [0,v^-]$ defined by $f_2(x)=x\wedge v^-$ is an embedding, so $v=0$.
It follows that $r$ is strict.
\end{proof}

\begin{exm}\label{example}
(i) Let $S$ be a non-empty subset and for each $i\in S$, let $A_i$ be the finite MV-algebra $M_n=\{0,1/n,\ldots,n/n\}$, where $n\in\mathbb N$. By Theorem 4.4, the MV-algebra $\prod_{i\in S}A_i$ is Boolean subdirectly irreducible, but it has no square root.

(ii) Each linearly ordered pseudo MV-algebra with more than two elements is Boolean subdirectly irreducible. In addition, a direct product of any family of such pseudo MV-algebras is Boolean subdirectly irreducible.
\end{exm}

\begin{prop}\label{n2}
A representable pseudo MV-algebra $(M;\oplus,^-,^\sim,0,1)$ is Boolean subdirectly irreducible \iff $\bigcap X(M)_2=\{0\}$.
\end{prop}

\begin{proof}
Let $M$ be a representable pseudo MV-algebra.
If $M$ is Boolean subdirectly irreducible, then the subdirect embedding $f:M\to \prod_{P\in X(M)}M/P$ implies that
$f_2:M\to \prod_{P\in X(M)_2}M/P$ defined by $f_2(x)=(x/P)_{P\in X(M)_2}$ is an embedding, so
$\bigcap X(M)_2=\bigcap_{P\in X(M)_2}\Ker(\pi_P\circ f)=\Ker(f_2)=\{0\}$.

Conversely, assume that
$\bigcap X(M)_2=\{0\}$. Let $g:M\to \prod_{i\in I}M_i$ be a subdirect embedding. Without loss of generality, we can assume that each $M_i$ is linearly ordered.
Let $I_1=\{i\in I\mid M_i \mbox{ is a Boolean algebra}\}$ and $I_2=I\setminus I_1$.
Consider the map $g_2:M\to \prod_{i\in I_2}M_i$ defined by $g_2(x)=(\pi_i\circ g(x))_{i\in I_2}$.
Clearly, $\Ker(g_2)=\bigcap_{i\in I_2}\Ker(\pi_i\circ g)$.
For each $i\in I_2$, $M/\Ker(\pi_i\circ g)\cong M_i$, so $\Ker(\pi_i\circ g)\in X(M)_2$, consequently, $\bigcap X(M)_2\s \bigcap_{i\in I_2}\Ker(\pi_i\circ g)$.
Similarly, $\Ker(\pi_i\circ g)\in X(M)_1$ for all $i\in I_1$ which implies that $\bigcap X(M)_1\s \bigcap_{i\in I_1}\Ker(\pi_i\circ g)$.
Since $(\bigcap_{i\in I_2}\Ker(\pi_i\circ g))\cap (\bigcap_{i\in I_1}\Ker(\pi_i\circ g))=\{0\}$, we have $(\bigcap_{i\in I_2}\Ker(\pi_i\circ g))\cap \bigcap X(M)_1=\{0\}$.
We claim that $\bigcap X(M)_2=\bigcap_{i\in I_2}\Ker(\pi_i\circ g)$.
If there exists $a\in (\bigcap_{i\in I_2}\Ker(\pi_i\circ g))\setminus \bigcap X(M)_2$, then there exists $P\in X(M)_2$ such that $a\notin P$.
For each $b\in \bigcap X(M)_1$, $a\wedge b\in (\bigcap_{i\in I_2}\Ker(\pi_i\circ g))\cap \bigcap X(M)_1=\{0\}\s P$, so $b\in P$, that is $\bigcap X(M)_1\s P$.
By definition, $X(M)_1=\{Q\in X(M)\mid M/Q \mbox{ is a Boolean algebra}\}$, so each element of $X(M)_1$ is a Boolean ideal of $M$, whence $\bigcap X(M)_1$ is a
Boolean ideal.
Hence, $P$ is a Boolean ideal of $M$ (since $\bigcap X(M)_1\s P$) which contradicts with $P\in X(M)_2$. Therefore, $\{0\}=\bigcap X(M)_2=\bigcap_{i\in I_2}\Ker(\pi_i\circ g)=\Ker(g_2)$ implying
$g_2$ is one-to-one and $M$ is Boolean subdirectly irreducible.
\end{proof}

\begin{cor}\label{n4}
If  $(M;\oplus,^-,^\sim,0,1)$ is a representable pseudo MV-algebra, then $M/I_2$ is  Boolean subdirectly irreducible.
\end{cor}

\begin{proof}
Clearly, $M/I_2$ is representable. Let $Q\in X(M/I_2)_2$. Then $Q=P/I_2$ for some prime ideal $P$ of $M$ containing $I_2$.
Since $\frac{M/I_2}{P/I_2}\cong M/P$, we get $M/P$ is not a Boolean algebra, which entails that $P\in X(M)_2$. Similarly, the converse also holds.
Whence $X(M/I_2)_2=\{P/I_2\mid P\in X(M)_2\}$ and so $\bigcap X(M/I_2)_2=I_2/I_2=\{0/I_2\}$. By Proposition \ref{n2},
$M/I_2$ is Boolean subdirectly irreducible.
\end{proof}

\begin{thm}\label{n5}
Let $(M;\oplus,^-,^\sim,0,1)$ and $(N;\oplus,^-,^\sim,0,1)$ be pseudo MV-algebras with square roots $r$ and $s$, respectively, and let $f:M\to N$ be a homomorphism of pseudo MV-algebras. If $v=r(0)^-\odot r(0)^-$, $w=s(0)^-\odot s(0)^-$, $M_1$, $M_2$, $N_1$ and $N_2$ be the pseudo MV-algebras defined in \cite[Thm 4.3]{DvZa3}, then the following statements hold:
\begin{itemize}[nolistsep]
\item[{\rm (i)}] $f(M_2)\s N_2$;
\item[{\rm (ii)}] $f(M_1)$ is isomorphic to a subalgebra of $N_1\times A$, where $A=[0,f(v)\wedge w']$ is a Boolean  subdirectly irreducible subalgebra of $N_2$.
\end{itemize}
\end{thm}

\begin{proof}
According to \cite[Thm 4.3]{DvZa3}, assume $M_1=[0,v]$, $M_2=[0,v^-]$, $N_1=[0,w]$ and $N_2=[0,w^-]$.
From $f(r(0))\odot f(r(0))=f(0)=0$ it follows that $f(r(0))\leq s(0)$, so $s(0)^-\leq f(r(0)^-)$ which entails that
$w=s(0)^-\odot s(0)^-\leq f(r(0)^-)\odot f(r(0)^-)=f(v)$. Thus, $f(v^-)\leq w^-$.
Hence, $[0,f(v^-)]\s [0,w^-]=N_2$ and $[0,w]\s [0,f(v)]$.

Let $Q\in X(N)$. Since $Q$ is a proper prime ideal of $N$ and $f(v)\wedge f(v^-)=0\in Q$, only one of the statements $f(v)\in Q$ or $f(v^-)\in Q$ holds.
So, $X(N)$ is union of the disjoint sets $\{Q\in X(N)\mid f(v)\in Q\}$ and $\{Q\in X(N)\mid f(v^-)\in Q\}$.
On the other hand, $f(v)=(f(v)\wedge w)\vee (f(v)\wedge w^-)=w\vee (f(v)\wedge w^-)$, so
$[0,f(v)]\cong [0,w]\times [0,f(v)\wedge w^-]$.  It follows that $N\cong [0,w]\times [0,f(v)\wedge w^-]\times [0,f(v^-)]$.
Each set $[0,w]$, $[0,f(v)\wedge w^-]$ and $[0,f(v^-)]$ is a normal ideal of $N$ and so:

(i) If $x\in M_2$, then $f(x)\leq f(v^-)\leq w'$, so $f(x)\in [0,w^-]=N_2$. That is, $f(M_2)\s N_2$.

(ii) If $x\in M_1$, then $f(x)\leq f(v)$ and so $f(x)=(f(x)\wedge w)\vee (f(x)\wedge w^-)=(f(x)\wedge w)\vee (f(x)\wedge f(v)\wedge w^-)$.
Consider the injective homomorphism $g:M_1\to  N_1\times [0,f(v)\wedge w^-]$ defined by $g(x)=(f(x)\wedge w, f(x)\wedge f(v)\wedge w^-)$.
We show that $[0,f(v)\wedge w^-]$ is a Boolean  subdirectly irreducible subalgebra of $N_2$.
Since $f(v)\wedge w^-$ is a Boolean element of $N$, then clearly $[0,f(v)\wedge w^-]$ is a subalgebra of $[0,w^-]$.
Let $g:[0,f(v)\wedge w^-]\to B\times D$ be a subdirect embedding of pseudo MV-algebras, where $B$ is a Boolean algebra and $D$ is a pseudo MV-algebra.

We claim that $\pi_2\circ g:[0,f(v)\wedge w^-]\to D$ is one-to-one, equivalently, $g^{-1}(B\times\{0\})=\{0\}$.

Since $w^-=(f(v)\wedge w^-)\vee (f(v^-)\wedge w^-)=(f(v)\wedge w^-)\vee f(v^-)$ and $f(v),f(v^-),w^-\in \B(N)$, we get that
$[0,w^-]\cong[0,f(v)\wedge w^-]\times [0,f(v^-)]$ and so $\varphi:[0,w^-]\to (B\times D)\times [0,f(v^-)]$ defined by
$\varphi(x)=(g(x\wedge f(v)\wedge w^-),x\wedge f(v^-))$ is a subdirect embedding. Due to Theorem \ref{n6},
$[0,w^-]$ is Boolean subdirectly irreducible, so by (\ref{ne1}), $\varphi^{-1}(B\times \{0\}\times \{0\})=\{0\}$, consequently, by definition of $\varphi$,
$g^{-1}(B\times \{0\})=\{0\}$. So, the claim is true, and $[0,f(v)\wedge w^-]$ is a Boolean  subdirectly irreducible subalgebra of $N_2$.
\end{proof}

Now, we are ready to propose the main theorems of this part that will provide a square root closure (not necessarily strict) for MV-algebras. First, we state the following remark.

\begin{rmk}\label{n7}
Let $f:M\to N$ be an injective homomorphism of MV-algebras and $s:N\to N$ be a square root of $N$.
Let $I_1=\bigcap X(M)_1$ and $I_2=\bigcap X(M)_2$.
If $w=s(0)^-\odot s(0)^-$ and $N_1$ and $N_2$ be the MV-algebras that were introduced in \cite[Thm 4.3]{DvZa3}, then:

(i) The set $\{f^{-1}(Q)\mid Q\in X(N)\}$ represents each proper ideal $I$ of $M$, that is,
for each proper ideal $I$ of $M$, we have $I=\bigcap\{f^{-1}(Q)\mid Q\in X(N) \mbox{ and } f(I)\s Q\}$.
Choose a proper ideal $I$ of $M$. Clearly, $\downarrow f(I)=\{y\in N\mid y\leq f(x),~\exists x\in I\}$
is the least ideal of $N$ containing $f(I)$. It follows that $\downarrow f(I)=\bigcap\{Q\in X(N)\mid \downarrow f(I)\s Q\}=\bigcap\{Q\in X(N)\mid f(I)\s Q\}$ and so
$f^{-1}(\downarrow f(I))=\bigcap\{f^{-1}(Q)\mid Q\in X(N) \mbox{ and }f(I)\s Q\}$. On the other hand, if $x\in f^{-1}(\downarrow f(I))$, then $f(x)\leq f(a)$ for some $a\in I$,
whence $x\leq a$, since $f(x)=f(x\wedge a)$ and $f$ is one-to-one. Thus, $x\in I$ and so $f^{-1}(\downarrow f(I))\s I$. Clearly, $I\s f^{-1}(\downarrow f(I))$. That is,
$I=f^{-1}(\downarrow f(I))$. Therefore, $I=\bigcap\{f^{-1}(Q)\mid Q\in X(N) \mbox{ and } f(I)\s Q\}$.

(ii) We divide $X(N)$ into three disjoint sets as follows:
\begin{eqnarray*}
X(N)_1,\quad Y_1=\{Q\in X(N)_2\mid f^{-1}(Q)\in X(M)_1\}, \quad Y_2=\{Q\in X(N)_2\mid f^{-1}(Q)\in X(M)_2\}.
\end{eqnarray*}
If $Q\in X(N)_1$, then $f^{-1}(Q)\in X(M)_1$, since $M/f^{-1}(Q)$ is embedded in $N/Q$ and $N/Q$ is a Boolean algebra.
So, $I_1=\bigcap X(M)_1\s \bigcap \{f^{-1}(Q)\mid Q\in X(N)_1\cup Y_1\}$. On the other hand, according to (i), $I_1=\bigcap\{f^{-1}(Q)\mid Q\in X(N) \mbox{ and } f(I_1)\s Q\}=\bigcap\{f^{-1}(Q)\mid Q\in X(N) \mbox{ and } I_1\s f^{-1}(Q)\}$.
For each $Q\in X(N)$, if $I_1\s f^{-1}(Q)$, then $f^{-1}(Q)$ is a Boolean ideal of $M$ and so $Q\in Y_1\cup X(N)_1$, whence
$\bigcap X(M)_1\s \bigcap \{f^{-1}(Q)\mid Q\in X(N)_1\cup Y_1\}\s \bigcap\{f^{-1}(Q)\mid Q\in X(N)_1\cup Y_1\mbox{ and } I_1\s f^{-1}(Q)\}=I_1$.
Hence
\begin{eqnarray*}
\label{901}
I_1=\bigcap \{f^{-1}(Q)\mid Q\in X(N)_1\cup Y_1\}=f^{-1}\big((\bigcap X(N)_1)\cap (\bigcap Y_1)\big).
\end{eqnarray*}
Furthermore, clearly, $I_2\s \bigcap\{f^{-1}(Q)\mid Q\in Y_2\}$.  By (i), $I_2=\bigcap\{f^{-1}(Q)\mid Q\in X(N) \mbox{ and } f(I_2)\s Q\}$.
So,
\begin{eqnarray*}
\label{902}
I_2=(\bigcap\{f^{-1}(Q)\mid f(I_2)\s Q\in Y_2\})\cap (\bigcap\{f^{-1}(Q)\mid f(I_2)\s Q\in X(N)_1\cup Y_1\}).
\end{eqnarray*}
It follows that $M/I_2$ can be embedded into $M/A_1\times M/A_2$ by the mapping $\alpha(x)=(x/A_1,x/A_2)$, where $A_2=\bigcap\{f^{-1}(Q)\mid f(I_2)\s Q\in Y_2\}$ and
$A_1=\bigcap\{f^{-1}(Q)\mid f(I_2)\s Q\in X(N)_1\cup Y_1\}$. Since $M/A_1$ is a Boolean algebra and by Corollary \ref{n4}, $M/I_2$ is Boolean subdirectly irreducible, so
$\pi_2\circ \alpha$ is one-to-one. It follows that $I_2=A_2= \bigcap\{f^{-1}(Q)\mid f(I_2)\s Q\in Y_2\}$, whence $\bigcap\{f^{-1}(Q)\mid Q\in Y_2\}\s I_2$.
Therefore,
\begin{eqnarray}
\label{903}
I_2=\bigcap\{f^{-1}(Q)\mid Q\in Y_2\}.
\end{eqnarray}
\end{rmk}

In Definition \ref{div6}, we introduced the concept of a strict square root closure, and in Theorem \ref{div7}, we showed that each MV-algebra has a strict square root closure. Now, we want to define also the concept of square root closure (not necessarily a strict one) for  MV-algebras.

\begin{defn}\label{n8}
Let $M$ be an MV-algebra. A {\em square root closure} for M is a triple $(D,r,i)$,
where $D$ is an MV-algebra, $r:D\to D$ is a square root, and $i : M\to D$ is an embedding of MV-algebras such that
if $N$ is an MV-algebra with square root $s$ and $j:M\to N$ is an injective homomorphism of MV-algebras, there exists a unique injective
homomorphism $\beta:D\to N$ satisfying the following conditions:
\begin{itemize}[nolistsep]
\item[{\rm (D1)}] $\beta\circ i = j$ ;
\item[{\rm (D2)}] $\beta(r(x)) \leq s(\beta(x))$ for all $x\in M$.
\end{itemize}
\end{defn}

\begin{rmk}\label{n9}
(i) Similarly to strict square root closures, we can show that if an MV-algebra $M$ has a square root closure, then it is unique up to isomorphism. So, if it exists, the square root closure of $M$ is denoted by $\mathbf D(M)$.
An MV-algebra is called {\it closed with respect to square root closure}, if $\mathbf D(M)\cong M$.

(ii) If $M$ is an MV-algebra with a square root $r$, then $(M,r,\id_M)$ is a square root closure of $M$.
Indeed, if $N$ is an MV-algebra with a square root $s$ and $j:M\to N$ is an injective homomorphism, then
for $\beta=j$ clearly (D1) holds. For establishing (D2), we have
$\beta(r(x))\odot \beta(r(x))=\beta(r(x)\odot r(x))=\beta(x)$, so $\beta(r(x))\leq s(\beta(x))$ for all $x\in M$.
Therefore, $(M,r,\id_M)$ is a square root closure of $M$. It follows that each Boolean algebra $B$ has a square root closure,
$(B,\id_B,\id_B)$, since the identity map is a square root on $B$. Hence, if $M$ has a square root, then
$M$ is closed with respect to square root closure. It shows an important difference between the ``strict square root closure'' and the ``square root closure''. According to Theorem \ref{div7}, the strict square closure of $B$, $\mathbf C(B)$, is two-divisible, so $\mathbf C(B)$ is not isomorphic to $B$. But $\mathbf D(B)\cong B$.
\end{rmk}

\begin{thm}\label{n90}
Let $M=\Gamma(G,u)$ be an MV-algebra such that $u/2\in M$ and let $\mathbf D(M)=(\Gamma(K,w),r,i)$ be its square root closure. Then $\mathbf D(M)$ is the strict square root closure of $M$.
\end{thm}

\begin{proof}
Assume that $N=\Gamma(H,v)$ is an MV-algebra with a square root $s$ and $j:M\to N$ is an embedding.
We have $w=i(u)=i(u/2)\oplus i(u/2)=i(u/2)+i(u/2)$, that is $w/2=i(u/2)$. If $y\in \mathbf D(M)$ is such that $y\odot y\leq 0$, then $(y-w+y)\vee 0=0$, which entails that $y-w+y\leq 0$, equivalently, $y\leq w/2$.
Now, from $w/2\odot w/2=0$ it follows that $w/2=r(0)$, and so $r$ is strict. By Theorem 4.4, $\mathbf D(M)$ is a two-divisible $\ell$-group and so
$r(x)=(x+w)/2=x/2+w/2$ for all $x\in \mathbf D(M)$; see \cite[Thm 5.2]{DvZa3}.

On the other hand, $s(y)=(y+v)/2=y/2+v/2$ for all $y\in N$, by \cite[Thm 5.2]{DvZa3}.
Due to Definition \ref{n8}, there exists an injective homomorphism $\beta:\mathbf D(M)\to N$ such that $\beta\circ i=j$. That is, (C1) holds.
Also, we have
$\beta(r(x))=\beta(x/2+w/2)=\beta(x/2\oplus w/2)=\beta(x/2)\oplus \beta(w/2)=\beta(x)/2\oplus \beta(w/2)=\beta(x)/2+\beta(w)/2
=\beta(x)/2+v/2= s(\beta(x))$. Thus, (C2) holds. Therefore, $\mathbf D(M)$ is a strict square root closure for $M$.
\end{proof}

Let $(M;\oplus,',0,1)$ be an MV-algebra, $I_1=\bigcap X(M)_1$, $I_2=\bigcap X(M)_2$ and $f:M\to M/I_1\times M/I_2$ be defined by $f(x)=(x/I_1,x/I_2)$, $x\in M$,
satisfying the condition
\begin{eqnarray}
\label{nn12} \exists a\in M \mbox{ such that }  f(a)=(1/I_1,0/I_2).
\end{eqnarray}
Then $a\in \B(M)$. Since $(a/I_1,a/I_2)=f(a)=(1/I_1,0/I_2)$, we have $a\in I_2$ and so
$[0,a]\s I_2$. Moreover, if $x\in I_2$, then $f(x)=(x/I_1,x/I_2)=(x/I_1,0/I_2)\leq f(a)$, so $x\leq a$. Hence $I_2=[0,a]$. A similar calculus shows that
$I_1=[0,a']$. Therefore, the element $a$ satisfying condition \eqref{nn12} is unique. Due to Remark 2.2, $[0,a]$ and $[0,a']$ are MV-algebras such that $M/I_1=M/[0,a']\cong [0,a]=I_2$ and $M/I_2=M/[0,a]\cong [0,a']=I_1$. It follows that $I_2$ is a Boolean algebra and $I_1$ is a Boolean subdirectly irreducible MV-algebra (see Corollary \ref{n4}).
Considering these notations and conditions, we present the following theorem.

\begin{thm}\label{n10}
Let an MV-algebra $(M;\oplus,',0,1)$ satisfy one of the following conditions
\begin{itemize}
\item[{\rm (i)}] $I_1=\{0\}$.
\item[{\rm (ii)}] $I_2=\{0\}$.
\item[{\rm (iii)}] $I_1\ne \{0\}\ne I_2$ and let \eqref{nn12} hold.
\end{itemize}
Then $M$ has a square root closure. Moreover, in the first case, $M$ is a Boolean algebra and $\mathbf D(M)\cong M$. In the second case, $\mathbf D(M)\cong \mathbf C(M)$, and in the third case, $\mathbf D(M)\cong [0,a]\times \mathbf C([0,a'])$ for a unique Boolean element $a\ne 0,1$ satisfying \eqref{nn12}.
\end{thm}

\begin{proof}
(i) If $I_1=\{0\}$, then according to the definition of $X(M)_1$, $g:M\to \prod_{P\in X(M)_1}M/P$ is an embedding into
a Boolean algebra, so $M$ is a Boolean algebra. Hence by Remark \ref{n9}(ii), $(M,\id_M,\id_M)$ is the square root closure of $M$.

(ii) If $I_2=\{0\}$, by Corollary \ref{n4}, $M$ is Boolean subdirectly irreducible.
Let $(\textbf{C}(M),i,r)$ be the strict square root closure of  $M$. We also claim that it is
the square root closure of $M$. Assume that $N$ is an MV-algebra with a square root $s$ and
$g:M\to N$ is an injective homomorphism of MV-algebras. By \cite[Thm 2.21]{Hol},
$N\cong [0,v]\times [0,v']$, where $N_1=[0,v]$ is a Boolean algebra and $N_2=[0,v']$ is an MV-algebra with strict square root, and $v=s(0)'\odot s(0)'$
(the isomorphism is $x\longmapsto (x\wedge v,x\wedge v')$).
Theorem \ref{n3} together with Proposition \ref{n2} imply that $\pi_2\circ g:M\to [0,v']$ is one-to-one.
Clearly, $\pi_2\circ g(x)=g(x)\wedge v'$ for all $x\in M$.
Since $[0,v']$ is strict, according to the definition of
$\textbf{C}(M)$, there exists a unique one-to-one homomorphism $\alpha: \textbf{C}(M)\to [0,v']$ such that $\alpha\circ i=\pi_2\circ g$.

Let $f:M\to N_1\times N_2$ be the injective homomorphism defined by $f(x)=(g(x)\wedge v,g(x)\wedge v')$ for all $x\in M$.
Assume that $\textbf{C}(N_1)$ is the strict square root closure of $N_1$ and consider $N_1$ as a subalgebra of $\textbf{C}(N_1)$.
Clearly, $\textbf{C}(N_1)\times N_2$ is a strict MV-algebra. It follows that there exists an embedding $\beta:\textbf{C}(M)\to \textbf{C}(N_1)\times N_2$ such that
$\beta(x)=f(x)=(g(x)\wedge v,g(x)\wedge v')$ for all $x\in M$.
\setlength{\unitlength}{1mm}
\begin{figure}[!ht]
\begin{center}
\begin{picture}(50,20)
\put(13,17){\vector(2,0){12}}
\put(8,14){\vector(0,-1){12}}
\put(24,14){\vector(-1,-1){12}}
\put(5,-2){\makebox(4,2){{ $\textbf{C}(N_1)\times N_2$}}}
\put(5,16){\makebox(4,2){{ $M$}}}
\put(28,16){\makebox(6,2){{ $\textbf{C}(M)$}}}
\put(15,19){\makebox(4,2){{ $i$}}}
\put(1,7){\makebox(4,2){{$f$}}}
\put(20,7){\makebox(4,2){{$\beta$}}}
\end{picture}
\caption{\label{fig100n} Homomorphisms between $M$, $\textbf{C}(M)$, and $\textbf{C}(N_1)\times N_2$}
\end{center}
\end{figure}
Choose $x\in \textbf{C}(M)$. By Theorem \ref{th:crit}(iii), $x=\sum_{i=1}^{2^n}x_i/2^n$ for some $n\in\mathbb N$ and $x_1,\ldots,x_{2^n}\in M$.
We have
\begin{eqnarray*}
\label{k1} \beta(x)&=&\beta(\sum_{i=1}^{2^n}x_i/2^n)= (\pi_1\circ\beta(\sum_{i=1}^{2^n}x_i/2^n), \pi_2\circ\beta(\sum_{i=1}^{2^n}x_i/2^n))\\
\label{k2} &=& (\sum_{i=1}^{2^n}\pi_1\circ\beta(x_i/2^n), \sum_{i=1}^{2^n}\pi_2\circ\beta(x_i/2^n))=
(\sum_{i=1}^{2^n}(\pi_1\circ\beta(x_i))/2^n, \sum_{i=1}^{2^n}(\pi_2\circ\beta(x_i))/2^n)\\
\label{k3} &=& (\sum_{i=1}^{2^n}(g(x_i)\wedge v)/2^n,\sum_{i=1}^{2^n}(g(x_i)\wedge v')/2^n)= (\sum_{i=1}^{2^n}(g(x_i)\wedge v)/2^n,\alpha(x)).
\end{eqnarray*}
For simplicity, we denote $\pi_i\circ\beta$ by $\beta_i$ for $i=1,2$.
Define $\varphi:\textbf{C}(M)\to N_1\times N_2$ by $\varphi(x)=(\bigvee_{i=1}^{2^n}(g(x_i)\wedge v),\beta_2(x))$. We will show that $\varphi$ is an
injective homomorphism.
It suffices to show that $\pi_1\circ\varphi:\textbf{C}(M)\to [0,v]$ is a homomorphism. By definition of $\beta_1$, we have $2^n\beta_1(x)=2^n(\sum_{i=1}^{2^n}(g(x_i)\wedge v)/2^n)=\sum_{i=1}^{2^n}(g(x_i)\wedge v)$, whence
$(2^n\beta_1(x))\wedge v=\bigoplus_{i=1}^{2^n}(g(x_i)\wedge v)=\bigvee_{i=1}^{2^n}(g(x_i)\wedge v)=\pi_1\circ\varphi(x)$; note that
$g(x_i)\wedge v\in N_1\s \B(N)$ for all integers $1\leq i\leq 2^n$. Hence, for all $x\in \textbf{C}(M)$ we have
\begin{eqnarray}
\label{k5} \pi_1\circ\varphi(x)=(2^n\beta_1(x))\wedge v=2^n.\beta_1(x).
\end{eqnarray}
Thus, $\pi_1\circ\varphi$ is well-defined. Moreover, $\varphi$ is one-to-one, since $\alpha$ is one-to-one.
In addition, $\varphi(0)=(0,0)$ and $\varphi(1)=(g(1)\wedge v,g(1)\wedge v')=(v,v')$.

Now, we show that it preserves $\oplus$ and $'$.

(a) From $\pi_1\circ\varphi(x')=\pi_1\circ\varphi(u-x)= \pi_1\circ\varphi(u-\sum_{i=1}^{2^n}x_i/2^{n})= \pi_1\circ\varphi(\sum_{i=1}^{2^n}(u-x_i)/2^{n})= \sum_{i=1}^{2^n}(g(u-x_i)/2^n)=
v-\sum_{i=1}^{2^n}g(x_i)/2^n=v-\pi_1\circ\varphi(x)= (\pi_1\circ\varphi(x))'$ it follows that $\pi_1\circ\varphi$ preserves the unary operation $'$.

Let $x=\sum_{i=1}^{2^n}x_i/2^n$ and $y=\sum_{j=1}^{2^m}y_j/2^n$ for some $n,m\in\mathbb N$, where $x_1,\ldots,x_{2^n},y_1,\ldots,y_{2^m}\in M$; by Corollary \ref{co:RDP}, such representations exist. Without loss of generality, we can assume that $n=m$. Otherwise, if e.g. $n<m$, then $x_i/2^n = 2^{m-n}x_i/2^m$, and every element $x_i$ will be now $2^{m-n}$-times. By (\ref{k5}), we get
\begin{eqnarray*}
\pi_1\circ\varphi(x\oplus y)&=&2^n.\beta_1(x\oplus y)=2^n.\beta(x)\oplus 2^n.\beta_1(y)=\pi_1\circ\varphi(x)\oplus \pi_1\circ\varphi(y).
\end{eqnarray*}
Therefore, $\varphi:\textbf{C}(M)\to N_1\times N_2$ is an injective homomorphism. By definition of $\varphi$, if $x\in M$, then
$\varphi(x)=(g(x)\wedge v,g(x)\wedge v')=f(x)$. Hence, $\textbf{C}(M)$ is a square root closure for $M$.

(iii) Let $I_1\neq\{0\}$ and $I_2\neq\{0\}$ and let \eqref{nn12} hold. By the note right before this theorem, $I_1=[0,a']$, $I_2=[0,a]$, and $a\in \B(M)\setminus\{0,1\}$.
We claim that $[0,a]\times \mathbf C([0,a'])$ is the square root closure of $M$, where $(\mathbf C([0,a']),i,r)$ is the strict square root of the MV-algebra $[0,a']$.
Let $r_1:[0,a]\to [0,a]$ be the identity map and $r_2$ be the strict square root on $(\mathbf C([0,a']),i,r)$. Then $(r_1,r_2)$ is a square root on  $[0,a]\times \mathbf C([0,a'])$.
Also, the map $f:M\to [0,a]\times \mathbf C([0,a'])$ defined by $f(x)=(x\wedge a,x\wedge a')$ is an injective homomorphism.
Assume that $N$ is an MV-algebra with a square root $s$,
$g:M\to N$ is an injective homomorphism of MV-algebras, and $v$ is the element introduced in (ii).
By Remark \ref{n7}(\ref{903}), $I_2=\bigcap\{g^{-1}(Q)\mid Q\in Y_2\}$, so by Proposition \ref{10.5}(iii),
$[0,v]=\bigcap X_2^N\s \bigcap Y_2=\downarrow g(I_2)=\downarrow g([0,a])$. That is, $v\leq g(a)$ consequently, $g(a')\leq v'$.
We have
\begin{eqnarray*}
g(a)&=&(g(a)\wedge v)\vee (g(a)\wedge v')=v\vee (g(a)\wedge v'), \\
v'&=& (g(a')\wedge v')\vee (g(a)\wedge v')=g(a')\vee (g(a)\wedge v').
\end{eqnarray*}
(1) By (i), for the Boolean algebra $[0,a]$, we know that $([0,a],f_1,r_1)$ is its square root closure as it is shown in
the commutative diagram in Figure \ref{fig1n}, where $f_1=r_1$ is the identity map on $[0,a]$ and $\alpha_1=g$.
\setlength{\unitlength}{1mm}
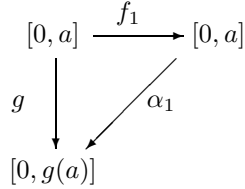
\begin{figure}[!ht]
\begin{center}
\begin{picture}(50,20)
\put(13,17){\vector(2,0){12}}
\put(8,14){\vector(0,-1){12}}
\put(24,14){\vector(-1,-1){12}}
\put(5,-2){\makebox(4,2){{ $[0,g(a)]$}}}
\put(5,16){\makebox(4,2){{ $[0,a]$}}}
\put(26,16){\makebox(6,2){{ $[0,a]$}}}
\put(15,19){\makebox(4,2){{ $f_1$}}}
\put(1,7){\makebox(4,2){{$g$}}}
\put(20,7){\makebox(4,2){{$\alpha_1$}}}
\end{picture}
\caption{\label{fig1n} The diagram of square root closure of $[0,a]$}
\end{center}
\end{figure}

(2) Furthermore, the MV-algebra $[0,a']$ is Boolean subdirectly irreducible.
Also, $[0,v']\cong [0,g(a')]\times [0,g(a)\wedge v']$ and $\varphi:[0,v']\to [0,g(a')]$ is an onto homomorphism, so by \cite[Prop 3.9]{DvZa3},
$t:[0,g(a')]\to [0,g(a')]$ defined by $t(\varphi(x))=\varphi(s(x))=s(x)\wedge g(a')$ for all $x\in [0,v']$ is a square root on $[0,g(a')]$.

On the other hand, by Theorem \ref{n6}, $[0,v']$ is Boolean subdirectly irreducible, so
$[0,g(a')]$ is Boolean subdirectly irreducible, too, because $[0,v']=[0,g(a')]\times [0,g(a)\wedge v']$. By Theorem \ref{n6}, $t$ is a strict square root on $[0,g(a')]$. According to the strict square root closure definition, we have the commutative diagram in Figure \ref{fig2n}.

\begin{figure}[!ht]
\begin{center}
\begin{picture}(50,20)
\put(13,17){\vector(2,0){12}}
\put(8,14){\vector(0,-1){12}}
\put(24,14){\vector(-1,-1){12}}
\put(5,-2){\makebox(4,2){{ $[0,g(a')]$}}}
\put(5,16){\makebox(4,2){{ $[0,a']$}}}
\put(26,16){\makebox(13,2){{ $\mathbf C([0,a'])$}}}
\put(15,19){\makebox(4,2){{ $f_2$}}}
\put(1,7){\makebox(4,2){{$g$}}}
\put(20,7){\makebox(4,2){{$\alpha_2$}}}
\end{picture}
\caption{\label{fig2n} The diagram of strict square root closures of $[0,a']$}
\end{center}
\end{figure}
From (1) and (2), it follows that $\alpha:[0,a]\times \mathbf C([0,g(a')])\to N$ defined by $\alpha(x,y)=\alpha_1(x)\vee \alpha_2(y)$
is a homomorphism of MV-algebras such that $\alpha\circ f=g$. Indeed, let $x\in M$. Then $\alpha(f(x))=\alpha(x\wedge a,x\wedge a')=\alpha_1(x\wedge a)\vee
\alpha_2(x\wedge a')=g(x\wedge a)\vee g(x\wedge a')=g(x)$. That is, (D1) holds.
Moreover, for each $(x,y)\in  [0,a]\times \mathbf C([0,g(a')])$ we have $\alpha((x,y))=\alpha((r_1,r_2)(x,y)\odot (r_1,r_2)(x,y))=\alpha((r_1,r_2)(x,y))\odot \alpha((r_1,r_2)(x,y))$
implies that $\alpha((r_1,r_2)(x,y))\leq s(\alpha((r_1,r_2)(x,y)))$, that is (D2) holds.
Therefore, $([0,a]\times \mathbf C([0,g(a')]),(r_1,r_2),f)$ is the square root closure of $M$.
\end{proof}

\begin{rmk}\label{rm:example}
(i) We note that due to Example \ref{example}(ii), every at least three-element linearly ordered MV-algebra is Boolean subdirectly irreducible. Due to Proposition \ref{n2}, $I_2=\bigcap X(M)_2=\{0\}$. Theorem \ref{n10}(ii) yields that it has a square root closure, and $\mathbf D(M)=\mathbf C(M)$.

(ii) Every direct product of at least three-element linearly ordered MV-algebras is Boolean subdirectly, i.e., $I_2=\{0\}$ and it has a square root closure, and $\mathbf D(M)=\mathbf C(M)$.

(iii) Let $M_n=\{1/n,2/n,\ldots,n/n\}$, $n\ge 2$. Then $X(M_n)_1=\emptyset$, $X(M_n)_2=\{\{0\}\}$, so $I_1=M$, $I_2=\{0\}$ and \eqref{nn12} holds with the element $a=0$. The same is true for $M=\Gamma(\mathbb R,1)$, and $\mathbf D(M)=\mathbf C(M)$

(iv) Let $M=\{0,1\}$. Then $X(M)_1=\{0\}$ and $X(M)_2=M$, and the element $a= 1$ satisfies \eqref{nn12}. Moreover, $\mathbf D(M)=\mathbf C(M)$.

(v) Let $M=\Gamma(\mathbb Z\lex \mathbb Z,(1,0))$. Then $X(M)_1=\{\{(0,0)\}\times \mathbb Z^+\}$, $X(M)_2=\{(0,0)\}$, $I_1=\{0\}\times \mathbb Z^+$, $I_2=\{(0,0)\}$, \eqref{nn12} fails. In addition, $\mathbf D(M)=\mathbf C(M)$.

(vi) Let $A=[0,1]$ be the standard MV-algebra of the real unit interval and $B$ an arbitrary Boolean algebra. For the MV-algebra $M:=B\times A$, we have
\begin{eqnarray*}
X(M)=\{P\times A\colon \mbox{ $P$ is a prime ideal of $B$}\}\cup \{B\times \{0\}\}.
\end{eqnarray*}
An ideal $I$ of $M$ belongs to $X(M)_1$ \iff $I=P\times A$ for some $P\in X(B)$. Therefore, $I_1=\bigcap X(M)_1=\{0\}\times A$ and $X(M)_2=\{B\times \{0\}\}$, so that $I_2=B\times \{0\}$. The element $a=(1,0)$ satisfies \eqref{nn12}.
\end{rmk}

Now, using (i) of the latter remark and the criterion Theorem \ref{th:crit}, we present some examples of the square root closures and their relation with the strict square root closures for $M_n=\{0,1/n,2/n,\ldots,n/n\}$ $(n\ge 1)$, $\Gamma(\mathbb Z\lex\mathbb Z,(1,0))$, and $M(\alpha)$, $\alpha$ rational in $[0,1]$. We define for each prime number $p\ge 2$,
\begin{align*}
\mathfrak{}\mathbb D(p)&=\big\{\frac{j}{p^m}\mid j\in \mathbb Z,\, m\ge 0\big\},\\
\mathbb D_2(p)&=\big\{\frac{1}{p}\frac{i}{2^k}\mid i\in \mathbb Z,\, k\ge 0\big\}.
\end{align*}
Then $(\mathbb D(p),1)$ is a unital $p$-divisible subgroup of $(\mathbb R,1)$ that is not two-divisible for any prime $p>2$, and
$(\mathbb D_2(p),1)$ is a unital two-divisible subgroup of $\mathbb R$ containing a copy of $\frac{1}{p}\mathbb Z$. In particular, $\mathbb D:=\mathbb D_2(1)$ is the group of dyadic numbers with strong unit $1$, and $\mathbb D(p)=\frac{1}{p}\mathbb D$.

\begin{exm}\label{ex:D(M)}
{\rm Let $M_n=\Gamma(\frac{1}{n}\mathbb Z,1)$ for $n\ge 1$.

(1) If $n=1$, $\mathbf D(M_1) =M_1$, $\mathbf C(M_1)=\Gamma(\mathbb D,1)$.

(2) If $n=2$, $M_2\varsubsetneq \mathbf C(M_2)=\Gamma(\mathbb D,1) = \mathbf D(M_2)$.

(3) If $n=p$ is a prime number, then $\mathbf D_p$ is not two-divisible, see \cite[Ex 4.15(2)]{DvZa2}, and $\mathbf D(M_p)=\Gamma(\mathbf D_2(p),1)=\mathbf C(M_p)$.

(4) Let $n=p_1^{j_1}\cdots p_k^{j_k}$ be a compound number, where each $p_j>1$ and $p_j$ is prime. Define
$$
\mathbf D_2(p_1,\ldots,p_k)=\frac{1}{p_1}\mathbb D+\cdots + \frac{1}{p_k}\mathbb D= \big\{\frac{1}{p^{j_1}_1}\frac{i_1}{2^{m_1}}+\cdots+ \frac{1}{p^{j_k}_k}\frac{i_k}{2^{m_k}}\mid i_1,\ldots,i_k\in \mathbb Z,\, m_1,\ldots,m_k\ge 0\big\}.
$$
Then $(\mathbf D_2(p_1,\ldots,p_k),1)$ is a unital subgroup of $(\mathbb R,1)$ that is two-divisible and containing a copy of $\frac{1}{p_i}\mathbb Z$ for each $i=1,\ldots,k$. So $M_n\varsubsetneq \mathbf D(M_n)=\Gamma(\mathbf D_2(p_1,\ldots,p_k),1)= \mathbf C(M_n)$.

(5) If $M=\Gamma(\mathbb Z\lex \mathbb Z,(1,0))$, then $M\varsubsetneq \mathbf D(M)=\Gamma(\mathbb D \lex \mathbb D,(1,0))=\mathbf C(M)$.

(6) For each irrational $\alpha$, $0<\alpha <1/2$, let $M(\alpha)=
\{m+n\alpha \mid m,n \in \mathbb Z, 0\le m+n\alpha\le 1\}$, $G(\alpha)=\{m + n\alpha \mid m,n\in \mathbb Z\}$, and $\mathbb D(\alpha)=\{d+m\alpha/2^n \mid d\in \mathbb D, n\ge 0, m\in \mathbb Z\}$. Then $M(\alpha)$ is an MV-subalgebra of $[0,1]$ generated by $\alpha$, and $M(\alpha) = \Gamma(G(\alpha),1) \varsubsetneq \mathbf D(M(\alpha))=\mathbf C(M(\alpha))= \Gamma(\mathbb D(\alpha), 1)$.

If $\beta$ is an irrational in $[0,1/2]$, $M(\alpha)\cong M(\beta)$ iff $M(\alpha)=M(\beta)$ iff $\alpha=\beta$; see \cite[Cor 7.2.6]{CDM}. Whence, we have uncountably many non-isomorphic MV-algebras $M(\alpha)$'s without square root; see \cite[Ex 4.15(1)]{DvZa2}.

(7) Let $M$ be the Boolean algebra $M= \{0,1\}^n$, $n\ge 1$, then $M\cong \Gamma(\mathbb Z^n,(1,\ldots,1))$ and $M=  \mathbf D(M)\varsubsetneq \Gamma(\mathbb D^n,(1,\ldots,1))=\mathbf C(M)$.

(8) Let $M=B$ be a Boolean algebra. Due to the Stone theorem, there is an algebra $\mathcal S(B)$ of subsets of some set $\Omega\ne \emptyset$ such that $B\cong \mathcal S(B)$. Define $\mathcal F(B)=\{\chi_A\mid A \in \mathcal S(B)\}$ and let $G(B)=\{\sum_{i=1}^jn_i\chi_{A_i}\mid n_i \in \mathbb Z, \chi_{A_i}\in \mathcal F(B), i=1,\ldots, j,\, j\ge 1\}$. Then $(G(B),1)$ is a unital $\ell$-group of integer-valued functions such that $B\cong \mathcal F(B)=\Gamma(G(B),1)$. If $\mathbb D(B)=\{f/2^n\mid f \in G(B), n\ge 0\}$, then $(\mathbb D(B),1)$ is a two-divisible unital $\ell$-group such that $B \cong \mathcal F(B)=\mathbf D(B)\varsubsetneq \mathbf C(B)=\Gamma(\mathbb D(B),1)$.

(9) Let $M_i=\Gamma(G_i,u_i)$ be an MV-algebra and let $\mathbf C(M_i)=\Gamma(D_i,u_i)$, where $(D_i,u_i)$ is a two-divisible unital $\ell$-group, $G_i\subseteq D_i$ for each $i=1,\ldots,n$. Then $M:=\prod_{i=1}^n M_i=\Gamma(G,u)$, where $u=(u_1,\ldots,u_n)\in G:=\prod_{i=1}^n G_i$. If $D:=D_1\times\cdots \times D_n$, then $D$ is a two-divisible $\ell$-group. Let $(h_1,\ldots,h_n)\in D$. Due to the criterion Theorem \ref{th:crit}, for each $h_i\in D_i$, there is an integer $m_i$ such that $2^{m_i}h_i\in G_i$. Let $m=m_1+\cdots+m_n$. Then $2^m(h_1,\ldots,h_n)=(2^mh_1,\ldots,2^mh_n)\in G$, which by the mentioned criterion gives $\mathbf C(\prod_{i=1}^n M_i)\cong \Gamma(D,u)= \prod_{i=1}^n \Gamma(D_i,u_i)= \prod_{i=1}^n \mathbf C(M_i)$.
}
\end{exm}

At the end of the section, we present the following problem.

\begin{open}
Let $(M;\oplus,',0,1)$ be an MV-algebra. Does it have the square root closure? If yes, is $M/I_1\times \mathbf C(M/I_2)$ the square closure of $M$?
\end{open}

\section{Greatest Subalgebras with Square Root}

In the last section, we embed special pseudo MV-algebras in pseudo MV-algebras with square roots. Now, we define a square root of an individual element of a pseudo MV-algebra and find the greatest subalgebra of a special pseudo MV-algebra with weak square root.

\begin{defn}\label{9.1}
An element $a$ of a pseudo MV-algebra $(M;\oplus,^-,^\sim,0,1)$ is called a {\it square root} of an element $x\in M$ if $a$ satisfies conditions (Sq1) and (Sq2). That is, $a\odot a=x$ and if $y\odot y\leq x$, then $y\leq a$ for all $y\in M$.
\end{defn}

Clearly, (Sq2) implies that if $x\in M$ has a square root $a\in M$, then $a$ is unique, so we denote the element $a$ by $\sqrt{x}$.
Moreover, in each pseudo MV-algebra, the top element $1$ has a square root and $\sqrt{1}=1$. Let $x\leq y$ be elements of $M$ such that $\sqrt{x}$ and $\sqrt{y}$ exist. Then $\sqrt{x}\odot\sqrt{x}\leq y$ and (Sq2) implies that $\sqrt{x}\leq \sqrt{y}$. Therefore,
\begin{eqnarray}
\label{eq9.1} x\leq y \mbox{ implies that } \sqrt{x}\leq \sqrt{y}.
\end{eqnarray}

If $\sqrt{x}$ exists for each $x\in M$, then the mapping $r:M\to M$ defined by $r(x)=\sqrt{x}$, $x\in M$, is a weak square root mapping, introduced in \cite[Def 3.2]{DvZa3}. If $M$ is an MV-algebra, then $r$ is a square root mapping introduced in \cite{Hol}; see also \cite{DvZa3}. Moreover, not every weak square root mapping on a pseudo MV-algebra is automatically a square root mapping (see \cite[Ex 6.1-6.2]{DvZa3}); on MV-algebras, they are equivalent. In those examples, $\sqrt{x}$ exists for each $x\in M$, but the mapping $x\mapsto \sqrt{x}$, $x\in M$, is not a square root; it is only a weak square root.

\begin{exm}\label{9.2}
(i) Consider the subalgebra $M_3=\{0,1/3,2/3,1\}$ of the MV-algebra of the real unit interval $[0,1]$.
Then $\sqrt{0}=1/3$, $\sqrt{1/3}=2/3$, $\sqrt{1}=1$, but $2/3$
does not have any square root. In the MV-algebra $M_{2n-1}=\Gamma(\frac{1}{2n-1}\mathbb Z,1)$, the square root of $1/(2n-1)$ exists and is equal to $n/(2n-1)$, whereas, the square root of $2(n-1)/(2n-1)$ fails. Similarly, in $M_{2n}=\Gamma(\frac{1}{2n}\mathbb Z,1)$, $\sqrt{0}=n/(2n)$ and $1/(2n)$ does not have a square root.

(ii) Consider the $\ell$-group $G:=\mathbb Z\lex\mathbb Z\lex\mathbb Z$ defined as follows: For $(a,b,c),(x,y,z)\in G$, let
\begin{eqnarray*}
(a,b,c)+(x,y,z)=(a+x,b+y,c+z+ay),\quad -(a,b,c,d)=(-a,-b,-c+ab).
\end{eqnarray*}
By \cite[P 138, E41]{Anderson}, $(G;+,(0,0,0))$ is a non-Abelian linearly ordered group. The element $u=(1,0,0)$ is a strong unit of $G$, and $M=\Gamma(G,u)$ is a pseudo MV-algebra.

The pseudo MV-algebra $\Gamma(G,u)$ has no square root. Let $m,n\in\mathbb N$. Consider the element
$(1,-2n,2m)\in M$. We have
$(1,-n,m)\odot (1,-n,m)=((1,-n,m)+(-1,0,0)+(1,-n,m))\vee (0,0,0)=((0,-n,m)+(1,-n,m))\vee (0,0,0)=(1,-2n,2m)\vee (0,0,0)=(1,-2n,2m)$.
On the other hand, assume that $(x,y,z)\in M$ such that $(x,y,z)\odot (x,y,z)\leq (1,-2n,2m)$. If $x=0$, then clearly
$(x,y,x)\leq (1,-n,m)$. If $x=1$, then $y\leq 0$ and  $(1,2y,2z)=(x,y,z)\odot (x,y,z)\leq (1,-2n,2m)$, so
$2y<-2n$ or $2y=-2n$ and $2z\leq 2m$. In both cases, $(x,y,z)\leq (1,-n,m)$. That is, $\sqrt{(1,-2n,2m)}=(1,-n,m)$.
\end{exm}

\begin{prop}\label{9.2.0}
Let $M=\Gamma(G,u)$ be a pseudo MV-algebra.
 \begin{itemize}[nolistsep]
 \item[{\rm (i)}] If $\sqrt{0}$ exists, then $\sqrt{0}\leq \sqrt{0}^-\wedge\sqrt{0}^\sim$ and $\sqrt{0}=\max\{x\wedge x^-\mid x\in M\}$.
 \item[{\rm (ii)}] The element $\sqrt{0}$ exists \iff the set $\{y\in M\mid 2y\leq u\}$ has a top element. In addition, if $M$ is representable and $u/2$ exists, then $\sqrt{0}$ exists and $\sqrt{0}=u/2$.
 \item[{\rm (iii)}] If $\sqrt{x}$ exists, then $\sqrt{x}\leq (x\oplus\sqrt{0})\wedge (\sqrt{0}\oplus x)$.
 \end{itemize}
\end{prop}

\begin{proof}
(i) We have $\sqrt{0}\odot \sqrt{0}=0$ and \cite[Prop 1.9]{georgescu} imply that $\sqrt{0}\leq \sqrt{0}^-,\sqrt{0}^\sim$. Then $\sqrt{0}=\sqrt{0}\wedge \sqrt{0}^-\in \{x\wedge x^-\mid x\in M\}$. For each $x\in M$,
$(x\wedge x^-)\odot (x\wedge x^-)\leq x\odot x^-=0$, so $x\wedge x^-\leq \sqrt{0}$. Therefore,
$\sqrt{0}=\max\{x\wedge x^-\mid x\in M\}$.

(ii) Since $\sqrt{0}-u+\sqrt{0}\leq(\sqrt{0}-u+\sqrt{0})\vee 0=\sqrt{0}\odot\sqrt{0}=0$, we have $\sqrt{0}-u\leq -\sqrt{0}$ and so
$2\sqrt{0}\leq u$, that is $\sqrt{0}\in \{y\in M\mid 2y\leq u\}$. Now, let $y\in M$ be such that $2y\leq u$, then $2y-u\leq 0$ and so $y-u\leq -y$ which implies that $y-u+y\leq 0$. Hence $y\odot y=(y-u+y)\vee 0=0$, consequently, $y\leq\sqrt{0}$.
Therefore, $\sqrt{0}=\max\{y\in M\mid 2y\leq u\}$. Now, assume that $M$ is representable and $u/2$ exists.
For each $y\in \{y\in M\mid 2y\leq u\}$, we have $2y\leq u=2(u/2)$, so $y\leq u/2$. Hence $u/2=\max\{y\in M\mid 2y\leq u\}=\sqrt{0}$.

(iii) From $(x^\sim\odot \sqrt{x})\odot (x^\sim\odot \sqrt{x})\leq x^\sim \odot (\sqrt{x}\odot\sqrt{x})\leq x^\sim\odot x=0$ and (Sq2), it follows that
$x^\sim\odot \sqrt{x}\leq \sqrt{0}$, whence $\sqrt{x}=x\vee\sqrt{x}=x\oplus (x^\sim\odot \sqrt{x})\leq x\oplus \sqrt{0}$.
In a similar way, using $(\sqrt{x}\odot x^-)\odot (\sqrt{x}\odot x^-)\leq (\sqrt{x}\odot\sqrt{x})\odot x^-\leq x\odot x^-$,
we can prove that $\sqrt{x}\leq \sqrt{0}\oplus x$.
\end{proof}

\begin{lem}\label{9.3}
Let $M=\Gamma(G,u)$ be a symmetric representable pseudo MV-algebra, $u/2\in M$, $N$ be a linearly ordered pseudo MV-algebra and $f:M\to N$ be a surjective homomorphism. If $\sqrt{x}$ exists for some $x\in M$, then $f(\sqrt{x})$ exists and
\begin{equation*}
f(\sqrt{x})=\left\{\begin{array}{ll}
(f(x)+f(u))/2=\sqrt{f(x)} &  \mbox{ if }  f(x)\neq 0\\
f(u)/2=\sqrt{0}& \ \mbox{if }  f(x)=0.
\end{array} \right.
\end{equation*}
\end{lem}

\begin{proof}
Since $u/2$ is defined in $M$, clearly, $f(u)/2$ is defined in $N$.
Let $x\in M$ be such that $\sqrt{x}$ exists. If $f(x)\neq 0$, then $0<f(x)=f(\sqrt{x}\odot \sqrt{x})=f(\sqrt{x})\odot f(\sqrt{x})=
(f(\sqrt{x})-f(u)+f(\sqrt{x}))\vee 0$, so $f(x)=f(\sqrt{x})-f(u)+f(\sqrt{x})$ (since $N$ is a chain and symmetric).
It follows that $f(x)+f(u)=2f(\sqrt{x})$, that is $f(\sqrt{x})=(f(x)+f(u))/2$. Clearly,
$f(\sqrt{x})\odot f(\sqrt{x})=f(x)$. Let $y\in N$ be such that $y\odot y\leq f(x)$.
Then $2y\leq f(x)+f(u)=2f(\sqrt{x})$, consequently, $y\leq f(\sqrt{x})$. Thus, $\sqrt{f(x)}=f(\sqrt{x})$.

If $f(x)=0$, by Proposition \ref{9.2.0}(ii), in $N$ we have, $\sqrt{0}=f(u)/2$ and so $f(u/2)=f(u)/2=\sqrt{0}=\sqrt{f(x)}$.
\end{proof}

\begin{prop}\label{9.2.1}
Let $M=\Gamma(G,u)$ be a symmetric representable pseudo MV-algebra such that $u/2\in G$ and let $x\in M$. Then
\begin{itemize}[nolistsep]
\item[{\rm (i)}] $\sqrt{0}$ exists and $\sqrt{0}=u/2\in \C(G)$.
\item[{\rm (ii)}] $\sqrt{x}$ exists \iff $(x+u)/2\in M$, in either case, $\sqrt{x}=(x+u)/2$. In either case, $x/2\in M$. Moreover, if $M$ is a subdirect product of \{$\Gamma(G_i,u_i)\mid i\in I\}$, and $x=(x_i)_i\in M$, then $\sqrt{x}$ exists \iff $\sqrt{x_i}$ exists in $\Gamma(G_i,u_i)$ for each $i\in I$. In either case, $\sqrt{x}=(\sqrt{x_i})_i$.
\item[{\rm (iii)}] If $\sqrt{x}$ exists, then $\sqrt{x^-}$ exists in $M$, and $\sqrt{x^-}=\sqrt{x}\ra \sqrt{0}=\sqrt{x^\sim}$.
\item[{\rm (iv)}] If $\sqrt{x}$ and $\sqrt{y}$ exist, then $\sqrt{x\vee y}$ and $\sqrt{x\wedge y}$ exists in $M$, $\sqrt{x\vee y}=\sqrt{x}\vee \sqrt{y}$, and $\sqrt{x\wedge y}=\sqrt{x}\wedge\sqrt{y}$.
\item[{\rm (v)}] Let $\sqrt{x}$ and $\sqrt{y}$ exist. If $x\oplus y=y\oplus x$, then  $\sqrt{x\oplus y}=(\sqrt{x}\odot\sqrt{0}^-)\oplus \sqrt{y}$.
 \item[{\rm (vi)}] Let $\sqrt{x}$ and $\sqrt{y}$ exist. If $x+ y=y+ x$, then  $\sqrt{x\odot y}=(\sqrt{x}\odot \sqrt{y})\vee\sqrt{0}$.
 \item[{\rm (vii)}] $\sqrt{z\odot z}$ and $\sqrt{z\oplus z}$ exist for all $z\in M$.

\end{itemize}
\end{prop}

\begin{proof}
Let $(G,u)$ be a representable unital $\ell$-group, $\{(G_i,u_i)\mid i\in I\}$ be a family of linearly ordered unital $\ell$-groups, and $f:G\to \prod_{i\in I}G_i$ be a subdirect embedding. By Proposition \ref{9.2.0}(ii), $\sqrt{0}=u/2$. Let $\pi_i:\prod_{i\in I}G\to G_i$ be the $i$-th natural projection map, $i\in I$.  We denote $\pi_i\circ f$ simply by $f_i$.

(i) Choose $g\in G$. For each $i\in I$, from $u_i=f_i(u)=f_i(u/2+u/2)=f_i(u/2)+f_i(u/2)$, it follows that $f_i(u)/2$ exists in $G_i$ and $f_i(u)/2=f_i(u/2)$. By Proposition \ref{9.2.0}(ii), $\sqrt{0_i}$ exists in the symmetric pseudo MV-algebra $\Gamma(G_i,f_i(u))$ and $\sqrt{0_i}=f_i(u/2)$.
Since $G_i$ is linearly ordered, we have:
If $f_i(g)+f_i(u/2)\leq f_i(u/2)+f_i(g)$, then $f_i(u/2)+f_i(g)+f_i(u/2)\le f_i(u)+f_i(g)=f_i(g)+f_i(u)$, so $f_i(u/2)+f_i(g)\leq f_i(g)+f_i(u/2)$. Analogously, we proceed if $f_i(u/2)+f_i(g)\le f_i(g)+f_i(u/2)$. Thus, $u/2\in \C(G)$.
Therefore, $\sqrt{0}$ exists, $\sqrt{0}=u/2$, and belongs to $\C(G)$.

(ii) Let $x=(x_i)_i\in M$ be such that $(x+u)/2$ exists. Then $(x+u)/2\odot (x+u)/2=x$ and if $t\odot t\leq x$, then $2t-u\leq (t-u+t)\vee 0=t\odot t\leq x$ implies that
$2t\leq x+u$ and so $t\leq (x+u)/2$. That is, $\sqrt{x}=(x+u)/2=((x_i+u_i)/2)_i$.

Conversely, let $\sqrt{x}$ exist. Choose $i\in I$.
Due to Lemma \ref{9.3}, we have
$$
f_i(\sqrt{x})=(f_i(x)+f_i(u))/2= \sqrt{f_i(x)}.
$$
So, $f_i(2\sqrt{x})=(f_i(x)+f_i(u))=f_i(x+u)$ for all $i\in I$, consequently $2\sqrt{x}=(x+u)$ and so $\sqrt{x}=(x+u)/2$.

In addition, $x/2=(x+u)/2-u/2=\sqrt{x}-u/2\le u-u/2\le u$, so that $x/2\in M$.

From these ideas, we have also $\sqrt{x}$ exists in $M$ iff $\sqrt{x_i}$ exists in $\Gamma(G_i,u_i)$, and then $\sqrt{x}=(\sqrt{x_i})_i$.

(iii) $\sqrt{x}\ra \sqrt{0}=\big((u-(x+u)/2)+u/2\big)\wedge u=(u-x/2)\wedge u=u-x/2=(x^-+u)/2$. Hence by (ii),
$\sqrt{x^-}$ exists and is equal to $\sqrt{x}\ra \sqrt{0}$. Since $M$ is symmetric, $x^-=x^\sim$ and $\sqrt{x^-}=\sqrt{x^\sim}$.

(iv) Let $\sqrt{x}$ and $\sqrt{y}$ be defined. For the first part, it suffices to show that $((x\vee y)+u)/2$ exists and is equal to $\sqrt{x}\vee\sqrt{y}$.
According to \cite[1.1.10]{Anderson}, we have $(x\vee y)+u=(x+u)\vee (y+u)$. Since $(x+u)/2$ and $(y+u)/2$ exist, so
$((x\vee y)+u)/2=(x+u)/2 \vee (y+u)/2$ exists. Therefore, by (ii),
$\sqrt{x\vee y}=\sqrt{x}\vee\sqrt{y}$.

For the second part, we can use analogous statements as in the first one, or we can use (iii) and the first part.

(v) Due to (iii), it is enough to assume that $M$ is linearly ordered. First, we note that if $x+y=y+x$, then $x/2+y/2=x/2+x/2$. Indeed, since $M$ is a chain assume that $x+y/2\leq y/2+x$, then $y+x=x+y\leq y/2+x+y/2$ and so
$y/2+x\leq x+y/2$. That is $x+y/2=y/2+x$.  Similarly, using  $x+y/2=y/2+x$, we can prove that $x/2+y/2=y/2+x/2$.

Check
$(\sqrt{x}\odot \sqrt{0}^-)\oplus \sqrt{y}=((x+u)/2\odot u/2)\oplus (y+u)/2=((x+u)/2-u+ u/2)\oplus (y+u)/2=(x/2+(y+u)/2)\wedge u$.
If $u\leq x+y$, then $\sqrt{x\oplus y}=\sqrt{u}=u$. On the other hand,  $u-y\leq x$ and so $u/2-y/2=(u-y)/2\leq x/2$ which implies that $u/2\leq x/2+y/2$. Thus, $u\leq x/2+y/2+u/2=x/2+(y+u)/2$, consequently $(\sqrt{x}\odot \sqrt{0}^-)\oplus \sqrt{y}=(x/2+(y+u)/2)\wedge u=u$,
that is $\sqrt{x\oplus y}=u=(\sqrt{x}\odot \sqrt{0}^-)\oplus \sqrt{y}$. If $x\oplus y<u$, then $y\oplus x<u$, so by the assumption we have $x+y=x\oplus y=y\oplus x=y+x$.
Hence $(\sqrt{x}\odot \sqrt{0}^-)\oplus \sqrt{y}=(x/2+(y+u)/2)\wedge u=((x+y+u)/2)\wedge u=((x+y)+u)/2=((x\oplus y)+u)/2=\sqrt{x\oplus y}$, since $x+y=y+x$.

(vi) Again, we can assume $M$ is linearly ordered. By (i), we have
\begin{eqnarray*}
(\sqrt{x}\odot\sqrt{y})\vee\sqrt{0}&=&\big(((x+u)/2-u+(y+u)/2)\vee 0\big)\vee u/2=(x+y)/2\vee u/2\\
&=&((x+y-u)/2\vee 0)+u/2=\big(((x+y-u)\vee 0)/2\big)+u/2\\
&=&((x\odot y)+u)/2.
\end{eqnarray*}
Hence by (iii), $\sqrt{x\odot y}=(\sqrt{x}\odot\sqrt{y})\vee\sqrt{0}$.
Note that since $G$ is linearly ordered, for each $x\in G$ with $x/2\in G$, we have $x/2\vee 0=(x\vee 0)/2$.

(vii) By (iii), we assume $M$ is linearly ordered. Let $z\in M$. We have $(z\vee \sqrt{0})\odot (z\vee \sqrt{0})=(z\odot z)\vee (z\odot \sqrt{0})$. If $z\leq \sqrt{0}$, then
$z\odot z=0=(z\odot z)\vee (z\odot \sqrt{0})$. If $\sqrt{0}<z$, then $z\odot \sqrt{0}\leq z\odot z$ and so
$(z\odot z)\vee (z\odot \sqrt{0})=z\odot z$.
Now, let $y\in M$ such that $y\odot y\leq z\odot z$. If $y\odot y=0$, then $y\leq \sqrt{0}$ and so $y\leq z\vee \sqrt{0}$.
Otherwise, $0<y\odot y$ implies that $y\odot y=y-u+y=2y-u$ and $z\odot z=z-u+z=2z-u$. Hence
$2y-u\leq 2z-u$, so $2y\leq 2z$ consequently $y\leq z\leq z\vee \sqrt{0}$. Therefore, $\sqrt{z\odot z}=z\vee \sqrt{0}$.

Since $\sqrt{z^-\odot z^-}$ exists, then by (iv), $\sqrt{z\oplus z}=\sqrt{(z^-\odot z^-)^\sim}$ exists.
\end{proof}

Let $X$ be a subalgebra of a pseudo MV-algebra $M$. We say that $X$ has the {\it weak square root property} if $\sqrt x$ exists in $M$ for each $x\in X$, $\sqrt x\in X$. In this case, the mapping $r_X(x)=\sqrt x$, $x\in X$, is a weak square root on $X$.

We note that the Boolean skeleton $\B(M)$ of a pseudo MV-algebra has the weak square root property iff $\sqrt 0$ exists in $M$ and $\sqrt 0 = 0$. Indeed, let $\sqrt 0=0$. The Boolean algebra $\B(M)$ has the unique square root $r_{\B(M)}=\id_{B(M)}$. Let $b \in \B(M)$ and let $y\in M$ be such that $y\odot y \le b$. Then $(y\wedge b^-)\odot (y\wedge b^-)\le 0=\sqrt 0$, so that $y\wedge b^-\le 0$ and $y \le b$. This implies $\sqrt b$ exists in $M$, $\sqrt b=b$, and $\B(M)$ has the weak square root property.

The converse statement is clear because the only square root mapping on a Boolean algebra is the identity map.

\begin{thm}\label{9.4}
Let $M=\Gamma(G,u)$ be a representable symmetric pseudo MV-algebra such that $u/2\in G$ and define $X_1=\{x\in M\mid \sqrt{x}\mbox{ exists}\}$.
\begin{itemize}[nolistsep]
\item[{\rm (i)}] If $M$ is an MV-algebra, then $X_1$ is a subalgebra of $M$ and $X_1$ has the weak square root property.
\item[{\rm (ii)}] If $M$ is an MV-algebra and $u/2^n$ exists for all $n\in\mathbb N$, then the greatest subalgebra of $M$ with weak square root exits.
\item[{\rm (iii)}] We have $x\odot x\in X_1$ for all $x\in M$.
\end{itemize}
\end{thm}

\begin{proof}
(i) By Proposition \ref{9.2.1}(i), $0\in X_1$, so $0,1\in X_1$ (recall that $\sqrt{1}=1$).
Also, by Proposition \ref{9.2.1}(iii) and (iv), $X_1$ is closed under the operations $\vee$, $\wedge$ and $'$.
In addition, since $\oplus$ is commutative, Proposition \ref{9.2.1}(v) implies that $X_1$ is closed under $\oplus$.
Therefore, $X_1$ is a subalgebra of $M$.

(ii) Let  $u/2^n$ exist for all $n\in\mathbb N$.
Set $X_{n+1}:=\{x\in X_n\mid \sqrt{x}\mbox{ exists}\}$ for all $n\in\mathbb N$. Using induction and reasoning  similar to (i), $X_{n+1}$ is a subalgebra of $X_n$ and so
$X:=\bigcap_{n\in\mathbb N}X_n$ is a subalgebra of $M$. Clearly, $X$ is a pseudo MV-algebra with square root. It is the greatest
subalgebra of $M$ with square root.

(iii) Let $x\in M$. Use a standard subdirect decomposition to linear factors. By Proposition \ref{9.2.1}(vi), $\pi_i(f(x))\vee \sqrt{0_i}$ is a square root of $\pi_i(f(x))\odot \pi_i(f(x))=\pi_i(f(x\odot x))$. It follows that for each $i\in I$, we have
\begin{eqnarray*}
\pi_i(f(x\odot x))&=&\big(\pi_i(f(x))\vee \sqrt{0_i} \big)\odot \big(\pi_i(f(x))\vee \sqrt{0_i} \big)\\
&=& \big(\pi_i(f(x))\vee \pi_i(f(\sqrt{0})) \big)\odot \big(\pi_i(f(x))\vee \pi_i(f(\sqrt{0})) \big) \quad \mbox{ by Lemma \ref{9.3}}\\
&=& \pi_i\circ f\big((x\vee \sqrt{0})\odot (x\vee \sqrt{0}) \big).
\end{eqnarray*}
Hence, $f(x\odot x)=f\big((x\vee \sqrt{0})\odot (x\vee \sqrt{0}) \big)$, and so
$x\odot x=(x\vee \sqrt{0})\odot (x\vee \sqrt{0})$. On the other hand, if $y\in M$ such that $y\odot y\leq x\odot x$, then
$\pi_i(f(y))\odot \pi_i(f(y))\leq \pi_i(f(x))\odot \pi_i(f(x))$. Proposition \ref{9.2.1}(vii) and Lemma \ref{9.3} entail
$\pi_i(f(y))\leq \pi_i(f(x))\vee \pi_i(f(\sqrt{0}))=\pi_i(f(x\vee \sqrt{0}))$ for all $i\in I$ which implies that
$f(y)\leq f(x\vee \sqrt{0})$ and so $y\leq x\vee \sqrt{0}$. Therefore, $\sqrt{x\odot x}$ exists and is equal to
$x\vee \sqrt{0}$. That is, $x\odot x\in X_1$.
\end{proof}

To illustrate the latter result, we have that if $M=M_3$, then $u/2\notin M$, and $X_1=\{0,1/3,1\}$ is not a subalgebra of $M$. The same is true for $M=M_{2n-1}$, where $u/2\notin M_{2n-1}$, $\sqrt{1/(2n-1)} = n/(2n-1)$, so $1/(2n-1)\in X_1=\{0,1/(2n-1), 3/(2n-1),\ldots,(2n-1)/(2n-1)\}$, but $(1/(2n-1))'=2(n-1)/(2n-1)\notin X_1$, and $X_1$ is not a subalgebra. On the other side, if $M=M_{2n}$, then $X_1=\{0,2/2n,4/2n,\ldots,2n/2n\}$ which is a subalgebra on $M$.

A close connection exists between the square root of $0$ and the square root of idempotent elements of a pseudo MV-algebra. Indeed, in a pseudo MV-algebra $M$, $\sqrt{0}$ exists \iff $\sqrt{b}$ exists for all $b\in\B(M)$:

\begin{prop}\label{9.5}
Let $(M;\oplus,^-,^\sim,0,1)$ be a pseudo MV-algebra such that $\sqrt{0}$ exists. Then:
\begin{itemize}[nolistsep]
\item[{\rm (i)}] If $\sqrt{x}$ exists, then $\sqrt{x}\leq (x\oplus\sqrt{0})\wedge (\sqrt{0}\oplus x)$.
\item[{\rm (ii)}] For all $b\in\B(M)$, $\sqrt{b}$ exists and $\sqrt{b}=b\vee \sqrt{0}$.
\item[{\rm (iii)}] If $\sqrt{0}=0$, then $\sqrt{b}=b$ for all $b\in\B(M)$. In addition, if $x\in M$ is such that $\sqrt{x}$ exists, then $x\in \B(M)$.
\end{itemize}
\end{prop}

\begin{proof}
(i) Let $x\in M$ and $\sqrt{x}$ exist. Then $\sqrt{x}\odot\sqrt{x}=x$ implies that
$(\sqrt{x}\odot \sqrt{x})\odot x^-=0$ and $x^\sim \odot (\sqrt{x}\odot \sqrt{x})=0$. Hence,
$(\sqrt{x}\odot x^-)\odot (\sqrt{x}\odot x^-)=0$ so
$\sqrt{x}\odot x^-\leq \sqrt{0}$ and $x^\sim\odot\sqrt{x}\leq \sqrt{0}$. Thus $\sqrt{x}= x\vee \sqrt{x}=x\oplus (x^\sim \odot \sqrt{x})\leq x\oplus \sqrt{0}$. Similarly, we can show that $\sqrt{x}\leq \sqrt{0}\oplus x$.

(ii) Let $b\in \B(M)$. The condition (Sq1) holds, because $(b\vee \sqrt{0})\odot (b\vee \sqrt{0})=b\vee (b\odot\sqrt{0})=b$.
From (i), we conclude that $\sqrt{b}=b\vee\sqrt{0}$.

(iii) It is a straight consequence of (i) and (ii).
\end{proof}

\begin{cor}\label{9.6}
Consider the assumptions of Theorem {\rm \ref{9.4}}. Then $\B(M)\s X$, where $X:=\bigcap_{n\in\mathbb N}X_n$ is the greatest subalgebra of $M$ with square root.
\end{cor}

\begin{proof}
It follows from Theorem \ref{9.4} and Proposition \ref{9.5}.
\end{proof}

\section{Conclusion}

Square roots on MV-algebras were generalized in \cite{DvZa3} to square roots on pseudo MV-algebras. In the present study, which is divided into two parts, we extended results on square roots and weak square roots.

Part I. We are interested in new characterizations of square roots with accent to strict square roots; see Theorem 3.3. We showed that a strict square root implies that the corresponding $\ell$-group is two-divisible, Theorem 4.1, Theorem 4.4, and Theorem 4.9. We characterized semisimple MV-algebras with strict square root; see Theorem 4.17.

Part II. We exhibited situations when a pseudo MV-algebra can be embedded into a pseudo MV-algebra with square root, see, e.g., Theorem \ref{div1}. We defined a strict square root closure of a pseudo MV-algebra, showing that every MV-algebra has a strict square root closure, Theorem \ref{div7}. We also introduced a square root closure of an MV-algebra, Theorem \ref{n10}, and we compared both notions of closures.
Finally, we have investigated the existence of square roots not for all elements but only for some. We showed that in some symmetric representable pseudo MV-algebra such that one-half of the top element exists, there is a maximal subalgebra with the weak square root property; see Theorem \ref{9.4}.

The research brought some interesting questions like (1) the existence of a strict square root closure for some pseudo MV-algebras that are not commutative, for example, for linearly ordered ones, and (2) the existence of a maximal subalgebra with square root that is not weak.

\end{document}